\documentclass{amsart}

\usepackage{amssymb}
\usepackage{graphicx}
\usepackage{amsthm}
\usepackage{float}
\usepackage{tikz}
\usepackage{url}
\usepackage{wrapfig}
\usepackage{bm}

\newcommand\ZZ{\mathbb{Z}}
\newcommand\ZZT{\mathbb{Z}[\sqrt{2}]}
\newcommand\ZZI{\mathbb{Z}[\sqrt{-1}]}
\newcommand\ZZJ{\mathbb{Z}[\sqrt{-2}]}

\newcommand\ZZd{\mathbb{Z}[\sqrt{d_0}]}

\newcommand\CC{\mathbb{C}}
\newcommand\RR{\mathbb{R}}
\newcommand\QQ{\mathbb{Q}}
\newcommand\QQd{\mathbb{Q}(\sqrt{d})}

\newcommand\Sd{\sqrt{d}}
\newcommand\Sdo{\sqrt{d_0}}
\newcommand\chipt{\chi_{\mathfrak{t}}}
\newcommand\chiw{\chi_{w}}

\newcommand\FF{\mathbb{F}}
\newcommand\OO{\mathcal{O}}

\newcommand\RRR{\mathcal{R}}
\newcommand\TTT{\mathcal{T}}
\newcommand\III{\mathcal{I}}
\newcommand\BB{\mathcal{B}}

\newcommand\oalpha{\overline{\alpha}}
\newcommand\ow{\overline{w}}
\newcommand\oz{\overline{z}}
\newcommand\bb{\mathfrak{b}}
\newcommand\obb{\overline{\mathfrak{b}}}

\newcommand\ff{\mathfrak{f}}

\newcommand\pp{\mathfrak{p}}

\newcommand\qq{\mathfrak{q}}

\newcommand\rrr{\mathfrak{r}}
\newcommand\pt{\mathfrak{t}}

\newcommand\aaa{\mathfrak{a}}
\newcommand\oaaa{\overline{\mathfrak{a}}}

\newcommand\ccc{\mathfrak{c}}

\newcommand\Aut{\mathrm{Aut}}

\newcommand\rk{\mathrm{rk}}
\newcommand\CL{\mathrm{Cl}}

\newcommand\Gal{\mathrm{Gal}}
\newcommand\Norm{\mathrm{N}}
\newcommand\val{\mathrm{ord}}

\newcommand\sign{\mathrm{sign}}
\newcommand\mul{\mathrm{m}}

\newcommand\Area{\mathrm{Area}}

\newcommand\ve{\varepsilon}

\newcommand\ONE{\textbf{1}}
\newcommand\br{\bm{r}}
\newcommand\bs{\bm{s}}
\newcommand\bc{\bm{c}}
\newcommand\be{\bm{e}}

\newcommand\BP{\mathbb{P}}

\theoremstyle{plain} \newtheorem{theorem}{Theorem}
\theoremstyle{plain} 
\theoremstyle{plain} 
\theoremstyle{plain} 
\theoremstyle{plain} \newtheorem{lemma}{Lemma}
\theoremstyle{plain} 
\theoremstyle{plain} 
\theoremstyle{plain} \newtheorem{conjecture}{Conjecture}
\theoremstyle{plain} 
\theoremstyle{plain} 
\theoremstyle{plain}\newtheorem{prop}{Proposition}
\theoremstyle{remark} \newtheorem*{remark}{Remark}
\theoremstyle{remark} 
\theoremstyle{remark} 
\theoremstyle{remark} 

\DeclareMathOperator*{\sumsum}{\sum\sum}

\DeclareMathOperator*{\suma}{\sum{}^{\ast}}

\title{On the $8$-rank of narrow class groups of $\QQ(\sqrt{-4pq})$, $\QQ(\sqrt{-8pq})$, and $\QQ(\sqrt{8pq})$}
\date{\today}
\author{Djordjo Milovic}
\address{University College London, London WC1E 6BT, United Kingdom}

\begin{document}

\begin{abstract}
Let $d \in \{-4, -8, 8\}$. We study the $8$-part of the narrow class group in the thin families of quadratic number fields of the form $\QQ(\sqrt{dpq})$, where $p\equiv q \equiv 1\bmod 4$ are prime numbers, and we prove new lower bounds for the proportion of narrow class groups in these families that have an element of order $8$. In the course of our proof, we prove a general double-oscillation estimate for the quadratic residue symbol in quadratic number fields.
\end{abstract}

\maketitle

\section{Introduction}
In \cite{Ste2}, Stevenhagen studied the $2$-part of narrow class groups in thin families of quadratic number fields parametrized by one prime number, namely families of the form $\{\QQ(\sqrt{dp})\}_{p\equiv 1(4)}$, where $d\in\{-4, -8, 8\}$ and where $p$ varies over prime numbers congruent to $1$ modulo $4$. In this paper, we aim to prove new results about the $2$-part of narrow class groups in similar thin families of quadratic number fields, except this time parametrized by products of two distinct prime numbers. We consider quadratic fields of the form $\QQ(\sqrt{dpq})$, where again $d\in\{-4, -8, 8\}$ and where now $p$~and~$q$ vary over pairs of distinct prime numbers congruent to $1$ modulo $4$. One of the key features of \cite{Ste2} (and more generally \cite{Ste1}) is that the distribution of the $8$-rank in one-parameter families as above can be deduced from the \v{C}ebotarev Density Theorem. The main novelty in the present setting is the introduction of double-oscillation estimates concerning certain families of Hecke characters to overcome the poor uniformity (in $q$) of the error terms in the \v{C}ebotarev Density Theorem when applied in families of number fields (parametrized by a prime $q$).

Let $d\in\{-4, -8, 8\}$, let $p\equiv q\equiv 1\bmod 4$ be two distinct prime numbers, and consider the narrow class group $\CL(dpq)$ of the quadratic number field of discriminant~$dpq$. We recall that the \text{narrow} class group of a number field $K$ is the quotient of the group of non-zero fractional ideals of $K$ by the subgroup of principal ideals that can be generated by an element $\alpha\in K$ satisfying $\sigma(\alpha)>0$ for every real embedding $\sigma: K\hookrightarrow \RR$; the narrow class group is canonically isomorphic, via the Artin map, to the Galois group of the maximal abelian extension of $K$ unramified at all finite primes. In particular, if $K$ is totally complex, then the narrow and the usual class groups coincide.

Given any finite group $G$ and an integer $k\geq 1$, we define the $2^k$-rank of $G$ to be $\rk_{2^k} G := \dim_{\FF_2}\left(2^{k-1}G/2^k G\right)$. Gauss's genus theory \cite{Gauss} then implies that $\rk_2\CL(dpq) = 2$, i.e., that the $2$-part of $\CL(dpq)$ is a direct sum of two cyclic $2$-groups. We wish to better understand of the size these cyclic $2$-groups. R\'{e}dei's work \cite{Redei} implies that $\rk_4 \CL(dpq) = 2$ if and only if $p\equiv q\equiv 1\bmod 8$ and $p$ is a square modulo~$q$. We note that Gerth~\cite{Gerth1, Gerth3, Gerth2} as well as Fouvry and Kl\a"{u}ners~\cite{FK2, FK0, FK1}, building on the work of Heath-Brown~\cite{HB1, HB2}, developed robust techniques to study the $4$-rank in families of many different types.

There are three main analytic results concerning the $8$-rank in families of quadratic number fields. First, Stevenhagen \cite{Ste1} proved that if $d\neq 0$ is any integer, then there is a normal extension $M_d/\QQ$ such that $\rk_8\CL(dp)$ is determined by the Artin symbol of $p$ in $M_d/\QQ$; hence the density of the set of primes $p$ for which $\rk_8\CL(dp)$ is equal to a given value can be deduced from the \v{C}ebotarev Density Theorem applied to $M_d/\QQ$. Note that the families studied by Stevenhagen are parametrized by a \textit{single} prime. Next, Fouvry and Kl\a"{u}ners \cite{FK1} proved certain distribution results about the $8$-rank in a special family parametrized by \textit{arbitrarily many} primes $\not\equiv 3\bmod 4$, but having $4$-rank equal to $1$. Finally, two recent works of Smith \cite{Smith1, Smith2} claim very strong results about the $8$- and higher $2$-power-ranks in the family of all \textit{imaginary} quadratic fields. His methods, however, heavily rely on the fact that the average number of prime factors of a discriminant $D$ grows as $\log\log D$ and are thus inapplicable to the \textit{thin} families we study. We prove
\begin{theorem}\label{mainthm}
Let $p$ and $q$ denote distinct prime numbers congruent to $1\bmod 4$. Then for $d\in\{-4, -8, 8\}$, we have
$$
\liminf_{X\rightarrow \infty}\frac{\#\{pq\leq X: \rk_4\CL(dpq) = 2, \rk_8\CL(dpq)\geq 1\}}{\#\{pq\leq X: \rk_4\CL(dpq) = 2\}}\geq \frac{c_d}{8},
$$
where $c_{-4} = c_{-8} = 2$ and $c_8 = 1$.
\end{theorem}
The asymptotic formula for the denominator in the ratio above is
\begin{equation}\label{basicdensity1}
\#\{pq\leq X: p\equiv q\equiv 1\bmod 4,\ \rk_4\CL(dpq) = 2\} \sim\frac{1}{32}\frac{X\log\log X}{\log X}
\end{equation}
as $X\rightarrow \infty$ (for any $d\in\{-4, -8, 8\}$). This formula is a slight variation of \cite[Equation (2.12), p. 493]{Gerth1}, whose proof for our particular case can be found in \cite{Gerth2}. The heuristic model of Cohen and Lenstra \cite{CohenLenstra} predicts that the limit in the Main Theorem exists and is equal to $5/8$ in the cases $d = -4$ and $d = -8$ and $11/32$ in the case $d = 8$. See Section \ref{Heuristics} for more details. We also note that the $16$- and higher $2$-power-ranks appear to be much harder to study from an analytic perspective, and there are only a few results in this direction \cite{Milovic1, Milovic2, KM1, KM2, Smith2}

The proof of the Theorem \ref{mainthm} exploits a new type of \textit{lower} bound for the $8$-rank. In \cite{FK1}, Fouvry and Kl\"{u}ners define a quantity $\lambda_D$ conducive to analytic techniques which gives a good \textit{upper} bound for the $8$-rank of the narrow class group $\CL(D)$ for a special class of positive discriminants $D$. This upper bound $\lambda_D$ actually coincides with $\rk_8\CL(D)$ when $\rk_4\CL(D) = 1$. However, when $\rk_4\CL(D) \geq 2$, the quantity $\lambda_D$ is only an upper bound for $\rk_8\CL(D)$ and hence cannot be used to deduce that $\rk_8\CL(D)\geq 1$. Therefore, Theorem~\ref{mainthm} cannot be readily deduced from the techniques in \cite{FK1}. One might try to deduce Theorem~\ref{mainthm} from~\cite{Ste1} by first applying, for a fixed $d\in\{-4, -8, 8\}$ and each prime $q\equiv 1\bmod 8$, the \v{C}ebotarev Density Theorem to the field extension $M_{dq}/\QQ$ to get a density $\delta_{d, q}$ for the set $S_{d, q}$ of primes $p\equiv 1 \bmod 4$ for which $\rk_4\CL(dpq) = 2$ and $\rk_8\CL(dpq)\geq 1$, i.e.,
$$
N_{d, q}(x) = \# \left\{p\in S_{d, q}: p\leq x \right\} = \delta_{d, q} \frac{x}{\log x} + E_{d, q}(x), 
$$
where $E_{d, q}(x) = o(x/\log x)$ as $x\rightarrow \infty$, and then patching together the contributions from different primes $q$ to get the asymptotics for the sum $\sum_{q\leq X} N_q\left(X/q\right)$ as $X\rightarrow \infty$. Unfortunately, the fields $M_{dq}$ are obtained via the existence theorem of class field theory and hence not explicit enough in $d$ and $q$ for this approach to work. Smith~\cite{Smith1}, following Corsman~\cite{Corsman}, constructs the fields $M_{dq}$ quite explicitly; however, the discriminants $dpq$ with $\rk_4\CL(dpq)=2$ are not \textit{generic} in the sense of Smith~\cite[Definition 2.4, p.\ 11]{Smith1} and hence not conducive to applying the \v{C}ebotarev Density Theorem. Perhaps more importantly, Smith assumes the Grand Riemann Hypothesis to overcome the very poor uniformity in $q$ of the best known bounds for the error term $E_{d, q}(x)$.

To avoid assuming the Grand Riemann Hypothesis, we prove double-oscillation results in quadratic rings (such as $\ZZI$, $\ZZJ$ and $\ZZT$) that are reminiscent of \cite[Proposition 21.3, p. 1027]{FI1}. In our case, however, we need somewhat more precise estimates -- the term $(MN)^{\epsilon}$ must be replaced by an arbitrary power of $\log {(MN)}$. A general approach to proving these types of double-oscillation results was already developed in \cite{Heilbronn}, so, after making appropriate adjustments to work inside more general number rings instead of the rational integers, the heart of the proof of \cite[Proposition 21.3, p. 1027]{FI1} lies in achieving cancellation in characters sums as in \cite[Lemma 21.1, p. 1025]{FI1}. In Proposition~\ref{keycancellation} of this paper, we give a shorter and more natural proof of a generalization of this result.

\subsection*{Acknowledgements}
I would like to thank Farrell Brumley, \a'{E}tienne Fouvry, Carlo Pagano, Peter Stevenhagen, and the anonymous referee for their useful advice. This research was supported by an ALGANT Erasmus Mundus Scholarship and National Science Foundation agreement No.\ DMS-1128155.

\section{Algebraic Criteria for the $8$-rank}\label{AlgebraicCriteria}
\subsection{Preliminaries}
Let $K$ be a quadratic number field of discriminant $D$, $\OO_K$ its maximal order, and $\CL$ the narrow class group of $\OO_K$. The \textit{narrow Hilbert class field} $H$ of $K$ is the maximal abelian extension of $K$ unramified at all finite primes. Hereafter, we will use the shorthand ``unramified a.f.p.'' for ``unramified at all finite primes''. The Artin map induces a canonical isomorphism of groups
\begin{equation}\label{ArtSym}
\left(\frac{\cdot}{H/K}\right):\CL\longrightarrow \Gal(H/K). 
\end{equation}
The above isomorphism allows us to deduce information about $\CL$ by constructing and studying abelian unramified a.f.p.\ extensions of $K$.

The $2$-torsion subgroup $\CL[2]$ is generated by the classes of the ramified finite primes in $K/\QQ$ (see for instance \cite[Corollary 9.9, p.\ 80]{Ste1}), i.e.,
\begin{equation}\label{twotorsion}
\CL[2] = \left\langle [\pp]:\pp\text{ prime ideal of }\OO_K\text{ such that }\pp|D\right\rangle.
\end{equation}
We will use the two facts above in tandem via the following lemma; although it is a straightforward generalization of the argument in \cite[p.\ 18-19]{Ste1}, we have not been able to find the exact statement in the literature, and so we include a proof for the sake of completeness. Hereafter, $C_n$ will denote a cyclic group of order $n$.
\begin{lemma}\label{timestwo}
Let $K$ be a quadratic number field. Suppose that $L/K$ is an unramified a.f.p.\ $C_{2^n}$-extension for some $n\geq 1$. Then every prime ideal $\pp$ of $\OO_K$ that is ramified in $K/\QQ$ splits completely in $L/K$ if and only if there exists an unramified a.f.p.\ $C_{2^{n+1}}$-extension $L'$ of $K$ containing $L$.
\end{lemma}
\begin{proof}
As $L/K$ is unramified a.f.p.\ and abelian, $L$ must be contained in the narrow Hilbert class field $H$. A prime ideal $\pp$ of $\OO_K$ splits completely in $L/K$ if and only if
$$
\left(\frac{\pp}{L/K}\right) = 1 \in \Gal(L/K)\cong \CL/\Gal(H/L).
$$
Hence, by \eqref{twotorsion}, every prime ideal $\pp$ of $\OO_K$ that is ramified in $K/\QQ$ splits completely in $L/K$ if and only if $\CL[2]\leq \Gal(H/L)$. Dually, in terms of the corresponding character groups, this holds if and only if
$$
\CL^{\vee}/\Gal(L/K)^{\vee}\cong \Gal(H/L)^{\vee} \twoheadrightarrow \CL[2]^{\vee}\cong \CL^{\vee}/\left(\CL^{\vee}\right)^2,
$$
i.e., if and only if $\Gal(L/K)^{\vee}\leq \left(\CL^{\vee}\right)^2$. Now, as $\Gal(L/K)\cong C_{2^n}$, so also $\Gal(L/K)^{\vee}\cong C_{2^n}$, with a generator, say, $\chi$. Thus $\Gal(L/K)^{\vee}\leq \left(\CL^{\vee}\right)^2$ if and only if $\chi = \psi^2$ for some $\psi\in \CL^{\vee}$, which holds if and only if there exists a group $A = \left\langle \psi\right\rangle\cong C_{2^{n+1}}$ with $\Gal(L/K)^{\vee}\leq A \leq \CL^{\vee}$. Dually, this holds if and only if there exists a $C_{2^{n+1}}$-extension $L'/K$ with $L\subset L'\subset H$. 
\end{proof}

We will also make use of the following lemma from Galois theory (see \cite[Chapter VI, Exercise 4, p.321]{Lang}). 
\begin{lemma}\label{lemNor}
Let $F$ be a field of characteristic different from $2$, let $E = F(\sqrt{d})$, where $d\in F^{\times}\setminus (F^{\times})^2$, and let $L = E(\sqrt{x})$, where $x\in E^{\times}\setminus (E^{\times})^2$. Let $N = \Norm_{E/F}(x)$. Then $N\in d\cdot (F^{\times})^2$ if and only if $L/F$ is normal with $\Gal(L/F)\cong C_4$, the cyclic group of order $4$.
\end{lemma}

\subsection{Special two-parameter families}
Let $d \in\{-4, -8, 8\}$, and let $p$ and $q$ be odd primes congruent to $1$ modulo $4$. Let $K = \QQ(\sqrt{dpq})$, and let $H$ denote its narrow Hilbert class field. Let $d_0 = d/4$, so that the maximal order of $K$ is $\OO_K = \ZZ[\sqrt{d_0pq}]$. We are ultimately interested in the average value of $\rk_8\CL(dpq)$ of $\OO_K$ as $p$ and $q$ range among prime numbers satisfying $pq\leq X$, for a real parameter $X$ going to infinity.

Let $\CL = \CL(dpq)$. Gauss's genus theory implies that $\rk_2\CL = 2$ and that the genus field, the maximal abelian extension of $\QQ$ contained in $H$, is
$$
G = H^{\CL^2} = \QQ(\Sd, \sqrt{p}, \sqrt{q}).
$$
The three quadratic subfields $G_1  = K(\Sd)$, $G_2 = K(\sqrt{p})$, and $G_3 = K(\sqrt{q})$ of $G$ correspond to the three proper subgroups of $\CL/\CL^2$. The three ramified primes $\pt$, $\pp$, and $\qq$ of $\OO_K$ that lie above $2$, $p$, and $q$, respectively, generate the $2$-torsion subgroup $\CL[2]$ and will play a prominent role in the subsequent discussions. Clearly $\rk_4\CL\leq\rk_2 \CL = 2$, and in fact the $4$-rank of $\CL$ is the largest it could be exactly when $p$ and $q$ satisfy
\begin{equation}\label{pqcong}
p\equiv q\equiv 1\bmod 8,
\end{equation}
and
\begin{equation}\label{pqjac}
\left(\frac{p}{q}\right) = 1.
\end{equation}
\begin{prop}\label{4rank}
Let $d\in\{-4, -8, 8\}$, and let $p$ and $q$ be odd prime numbers congruent to $1$ modulo $4$. Let $\CL = \CL(dpq)$ denote the narrow class group of the quadratic number field $\QQ(\sqrt{dpq})$. Then $\rk_4\CL = 2$ if and only if $p$ and $q$ satisfy \eqref{pqcong} and \eqref{pqjac}. 
\end{prop}
\begin{proof}
The extension $G_i/\QQ$ is a $V_4$-extension for $i = 1, 2, 3$, so the splitting behavior of $\pt$, $\pp$, and $\qq$ in $G_i/K$ is determined by the splitting behavior of $2$, $p$, and $q$, respectively, in quadratic subfields of $G_i$. Conditions \eqref{pqcong} and \eqref{pqjac} imply that $\pt$, $\pp$, and $\qq$ all split in $G_i/K$ for $i = 1, 2, 3$. For instance, by \eqref{pqcong}, the prime $p$ splits in $\QQd/\QQ$, and so $\pp$ splits in $G_1$. Now Lemma \ref{timestwo} implies that $p$ and $q$ satisfy \eqref{pqcong} and \eqref{pqjac} if and only if there exists an unramified a.f.p.\ $C_4$-extension $L_i/K$ containing $G_i$ for $i = 1, 2, 3$. As $\rk_2\CL = 2$, the result follows from applying Galois theory to the isomorphism \eqref{ArtSym}. 
\end{proof}
From now on, suppose $p$ and $q$ satisfy \eqref{pqcong} and \eqref{pqjac}. Although Proposition \ref{4rank} demonstrates the existence of at least three distinct unramified a.f.p.\ $C_4$-extensions of $K$, one for each $G_i$, it may be difficult to construct these extensions explicitly from $d$, $p$, and $q$. In one case, however, we can do exactly this.

By \eqref{pqcong}, both $p$ and $q$ split in the principal ideal domain $\ZZd$, so there exist primes $w, z\in \ZZd$ such that $\Norm(w) = p$ and $\Norm(z) = q$. If $d = -4$, then, again by \eqref{pqcong}, we have
$$
w, z\equiv \pm 1 \equiv \square\bmod 4\ZZI.
$$
If $d = -8$ or $d = 8$, then we can replace $w$ by $-w$ and/or $z$ by $-z$ if necessary to ensure that
$$
w, z\equiv 1\text{ or }3+2\sqrt{d_0} \equiv \square\bmod 4\ZZd.  
$$
In any case, we can choose primes $w$ and $z$ in $\ZZd$ such that 
\begin{equation}\label{wzconditions}
\Norm(w)= p,\ \ \Norm(z) = q,\ \ \text{and}\ \ w, z\equiv \square\bmod 4\ZZd.
\end{equation}
Define $\alpha\in\ZZd$ and $x, y\in\ZZ$ by the equation
\begin{equation}\label{alpha}
\alpha = wz = x+y\Sdo \in \QQd\subset G_1.
\end{equation}
Then $\alpha$ satisfies the condition
\begin{equation}\label{square}
\alpha \equiv \square\bmod 4\ZZd,
\end{equation}
and $p$, $q$, $x$, and $y$ satisfy the relation $pq = x^2 - d_0 y^2$. For an element $a$ in $\QQd$, we will denote the conjugate of $a$ in $\QQd/\QQ$ (or $G_1/K$) by $\overline{a}$, so that $\oalpha = \ow\oz = x-y\Sdo$. Let $L_1 = G_1(\sqrt{\alpha}) = K(\sqrt{d}, \sqrt{\alpha})$. Note that $\sqrt{\oalpha} = \pm\sqrt{dpq}/(\sqrt{d}\sqrt{\alpha})\in L_1$.
\begin{prop}\label{unramified}
Let $\alpha \in \QQd$ be given by \eqref{alpha}, and let $L_1$ be as above. Then $L_1/K$ is an unramified a.f.p.\ $C_4$-extension.
\end{prop}
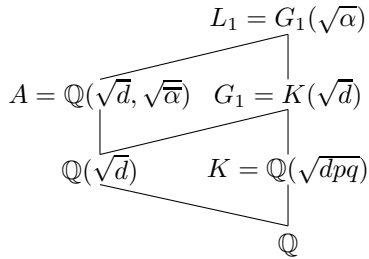
\begin{center} 
\begin{tikzpicture}
  \draw (0, 0) node[]{$L_1 = G_1(\sqrt{\alpha})$};
	\draw (0, -0.2) -- (0, -0.8);
  \draw (0, -1) node[]{$G_1 = K(\Sd)$};
	\draw (0, -1.2) -- (0, -1.8);
	\draw (0, -2) node[]{$K = \QQ(\sqrt{dpq})$};
	\draw (0, -2.2) -- (0, -2.75);
	\draw (0, -3) node[]{$\QQ$};
	\draw (-2.5, -1) node[]{$A = \QQ(\Sd, \sqrt{\oalpha})$};
	\draw (-2.5, -1.2) -- (-2.5, -1.8);
	\draw (-2.5, -2) node[]{$\QQd$};
  \draw (0, -0.2) -- (-2.5, -0.8);
	\draw (0, -1.2) -- (-2.5, -1.8);
	\draw (0, -2.75) -- (-2.5, -2.2);
\end{tikzpicture}
\end{center}
\begin{proof}
Since $\Norm_{G_1/K}(\alpha) = pq = d\cdot \left(\sqrt{dpq}/d\right)^2\in d\cdot (K^{\times})^2$, we see that $L_1/K$ is a $C_4$-extension, by Lemma \ref{lemNor}. The only primes that can ramify in $L_1/K$ are $\pt$, $\pp$, and $\qq$. We will show that $\pp$ is unramified in $L_1/K$, and by symmetry this will imply that $\qq$ is also unramified in $L_1/K$. As $w$ is a prime of degree one over $p$, it is coprime to~$\ow$. As $p$ and $q$ are distinct primes, $w$ is also coprime to $\oz$, and hence also to~$\oalpha$. Therefore $w$ does not ramify in $A = \QQ(\Sd, \sqrt{\oalpha})$, and so the ramification index of $p$ in $L_1/\QQ$ is at most $2$. But $p$ already ramifies in $K/\QQ$, and hence $\pp$ must be unramified in $L_1/K$. Finally, to see that $L_1/K$ is unramified over $\pt$, we may pass to the completion with respect to $\pt$ and show that $\QQ_2(\Sd, \sqrt{\alpha})/\QQ_2(\Sd)$ is unramified. This is the case if and only if $\alpha$ is a square modulo $4$ in the corresponding ring of integers $\ZZ_2[\sqrt{d_0}]$, and this is indeed ensured by condition \eqref{square}.\end{proof}

Now that we constructed $L_1/K$ explicitly, we can apply Lemma \ref{timestwo} to determine when $L_1$ is contained in an unramified a.f.p.\ $C_8$-extension $M_1/K$. We must determine when $\pt$, $\pp$, and $\qq$ all split completely in $L_1$. For the prime $\pt$, this can once again be determined locally. Indeed, $\pt$ splits completely in $L_1/K$ if and only if the extension of local fields $\QQ_2(\Sd, \sqrt{\alpha})/\QQ_2(\Sd)$ is trivial. This occurs if and only if $\alpha$ is a square in $\QQ_2(\sqrt{d})$, which happens if and only if $\alpha$ is a square modulo $\pt^5$, where by abuse of notation $\pt$ is now the maximal ideal in the discrete valuation ring $\ZZ_2[\Sdo]$. Explicitly, this means that 
\begin{equation}\label{split2}
\alpha\equiv \square\bmod \pt^5\equiv 
\begin{cases}
\pm 1 \bmod 4(1+\sqrt{-1})\ZZ_2[\sqrt{-1}] & \text{ if }d = -4,\\
1\text{ or }7+2\sqrt{-2} \bmod 4\sqrt{-2}\ZZ_2[\sqrt{-2}] & \text{ if }d = -8, \\
1\text{ or }3+2\sqrt{2} \bmod 4\sqrt{2}\ZZ_2[\sqrt{2}] & \text{ if }d = 8.
\end{cases}
\end{equation}
For primes $\pp$ and $\qq$, the splitting criterion is somewhat different. We may again use the auxiliary extension $A = \QQ(\Sd, \sqrt{\oalpha})$ from proof of Proposition \ref{unramified}. We have $p = w\ow$ with $w$ dividing $\alpha$, so $\pp$ splits completely in $L_1/K$ if and only if $w$ splits in $A/\QQd$. We use a quadratic residue symbol in $\QQd$ to detect this, i.e., $w$ splits in $A/\QQd$ if and only if
\begin{equation}\label{eqp}
\left(\frac{\oalpha}{(w)}\right) = 1.
\end{equation}
Similarly, the prime $\qq$ splits completely in $L_1/K$ if and only if
\begin{equation}\label{eqq}
\left(\frac{\oalpha}{(z)}\right) = 1.
\end{equation}
We will now explore the link between the quadratic residue symbols $\left(\frac{\oalpha}{(w)}\right)$ and~$\left(\frac{\oalpha}{(z)}\right)$. As $w$ and $z$ are primes of degree one over $p$ and $q$, respectively, we find that
$$
\left(\frac{\oalpha}{(w)}\right)\left(\frac{\oalpha}{(z)}\right) = \left(\frac{x-y\Sdo}{(x+y\Sdo)}\right) = \left(\frac{2x}{(x+y\Sdo)}\right) = \left(\frac{2x}{pq}\right), 
$$
where the last symbol is simply a Jacobi symbol. Using the fact that $pq = x^2 - d_0 y^2 \equiv 1\bmod 8$, we find that
$$
\left(\frac{2x}{pq}\right) = \left(\frac{|x|}{pq}\right) = \left(\frac{pq}{|x|}\right) = \left(\frac{x^2-d_0y^2}{|x|}\right) = \left(\frac{-d_0}{|x|}\right).
$$
We now make a distinction among the cases $d = -4$, $d = -8$, and $d = 8$. First suppose $d = -4$. Then $\left(\frac{-d_0}{|x|}\right) = \left(\frac{1}{|x|}\right) = 1$, and so
\begin{equation}\label{samecond}
\left(\frac{\oalpha}{(w)}\right) = \left(\frac{\oalpha}{(z)}\right).
\end{equation}
In other words, if $d = -4$, then $\pp$ splits completely in $L_1/K$ if and only if $\qq$ does. Therefore,  if $d = -4$, to ensure that $L_1$ is contained in an unramified a.f.p.\ $C_8$-extension $M_1/K$, we only need to verify that \eqref{split2} and \eqref{eqp} are satisfied. Next suppose $d = -8$. Then $\left(\frac{-d_0}{|x|}\right) = \left(\frac{2}{|x|}\right)$, and so
$$
\left(\frac{\oalpha}{(w)}\right) = \left(\frac{\oalpha}{(z)}\right) \Longleftrightarrow |x|\equiv 1, 7 \bmod 8 \Longleftrightarrow x\equiv 1, 7 \bmod 8,
$$
and this is guaranteed by \eqref{split2}. Again we conclude that $L_1$ is contained in an unramified a.f.p.\ $C_8$-extension $M_1/K$ provided that \eqref{split2} and \eqref{eqp} are satisfied. Finally, suppose $d = 8$. Then $\left(\frac{-d_0}{|x|}\right) = \left(\frac{-2}{|x|}\right)$, and so
$$
\left(\frac{\oalpha}{(w)}\right) = \left(\frac{\oalpha}{(z)}\right) \Longleftrightarrow |x|\equiv 1, 3 \bmod 8.
$$
Thus if $|x|\equiv 5, 7\bmod 8$, there is no chance that both \eqref{eqp} and \eqref{eqq} are satisfied and so $L_1$ is \textit{not} contained in an unramified a.f.p.\ $C_8$-extension $M_1/K$. Looking back at \eqref{split2}, we see that $\pt$ splits in $L_1/K$ if and only if $x$ satisfies
$$
x\equiv 1, 3 \bmod 8.
$$
Hence, assuming that \eqref{split2} holds, we find that $\left(\frac{\oalpha}{(w)}\right) = \left(\frac{\oalpha}{(z)}\right)$ if and only if
\begin{equation}\label{xpos}
x>0.
\end{equation}
As $x^2-2y^2 = pq$, we deduce that $|x|>|y\sqrt{2}|$, so that
$$
x>0 \Longleftrightarrow \alpha, \oalpha>0.
$$
In other words, the field $L_1$ cannot be contained in an unramified a.f.p.\ $C_8$-extension $M_1/K$ unless $L_1$ is totally real, i.e., unless $L_1/K$ is unramified also at the infinite places. We summarize the results of this section in the following proposition.
\begin{prop}\label{algebra}
Let $d\in\{-4, -8, 8\}$, and let $p$ and $q$ be prime numbers satisfying \eqref{pqcong} and \eqref{pqjac}. Let $w$ and $z$ be primes in $\ZZd$ satisfying \eqref{wzconditions}. Let $\alpha$ and $x$ be defined as in \eqref{alpha}. Suppose $\alpha$ satisfies \eqref{split2}. Furthermore, if $d = 8$, also suppose $x$ satisfies \eqref{xpos}. Then there is an unramified a.f.p.\ $C_8$-extension of $\QQ(\sqrt{dpq})$ containing $\QQ(\sqrt{d}, \sqrt{pq})$ if and only if
$$
\left(\frac{\oalpha}{(w)}\right) = 1.
$$
Consequently, under the assumptions above, if the equality above holds, then
$$
\rk_8\CL(dpq)\geq 1.
$$
\end{prop}

\subsection{The splitting condition for $\pt$}\label{sectionsplit2}
We now delve a bit deeper into the meaning of condition \eqref{split2}. Let $w$ and $z$ be primes in $\ZZd$ satisfying \eqref{wzconditions}, and let $\pt$ be the prime ideal of $\ZZd$ lying above $2$. Let $t$ be a generator of $\pt$ defined by
\begin{equation}\label{deft}
t = 
\begin{cases}
1+i & \text{ if } d = -4 \\
\sqrt{2} & \text{ if } d = 8.
\end{cases}
\end{equation}
In \cite[proof of Theorem 1, p. 5]{Ste2}, Stevenhagen proved that
\begin{equation}\label{SteResult}
\left(\frac{t}{(w)}\right) = 
\begin{cases}
1 &\text{ if } w\equiv \square\bmod \pt^5\\
-1 &\text{ otherwise},
\end{cases}
\end{equation}
and likewise for $z$. If we define 
\begin{equation}\label{chipt}
\chipt(\aaa) = \left(\frac{t}{\aaa}\right)
\end{equation}
for odd prime ideals $\aaa$ in $\ZZd$ and extend multiplicatively to the group $\III(\pt)$ of fractional ideals of $\ZZd$ coprime to $\pt$, then $\chipt$ is a quadratic Hecke character on $\ZZd$. The significance of \eqref{SteResult} is twofold: first, the variables $p$ and $q$, which are inextricably linked in the definition of $\alpha$, are now separated; and second, condition \eqref{split2} can now be written in terms of the quadratic Hecke character $\chipt$ on $\ZZd$, i.e.,
\begin{equation}\label{split22}
\alpha\equiv \square\bmod \pt^5 \Longleftrightarrow \chipt((w))\chipt((z)) = 1.
\end{equation}

\subsection{Positivity condition on $x$}
The variables $p$ and $q$ are also inextricably linked in the definition of variable $x$ appearing in \eqref{alpha}. However, the positivity condition \eqref{xpos} on $x$ can be unfolded via a theorem of Fouvry and Kl\"{u}ners \cite[Proposition 6, p.2063]{FK0}. We have
$$
x>0 \Longleftrightarrow [2, pq]_4 = [pq, 2]_4,
$$
where $[\cdot, \cdot]_4$ is the symbol defined in \cite[p. 2061]{FK0}, i.e.,
$$
[2, pq]_4 = [2, p]_4[2, q]_4,
$$
where
$$
[2, p]_4=
\begin{cases}
1 &\text{if }2\text{ is a fourth power modulo }p \\
-1 &\text{if }2\text{ is a square, but not a fourth power modulo }p \\
0 &\text{otherwise}
\end{cases}
$$
and similarly for $[2, q]_4$, and
$$
[pq, 2]_4=
\begin{cases}
1 &\text{if }pq\equiv 1\bmod 16 \\
-1 &\text{if }pq\equiv 9\bmod 16 \\
0 &\text{otherwise.}
\end{cases}
$$
When $p\equiv q\equiv 1\bmod 8$, then
$$
[2, pq]_4 = [2, p]_4[2, q]_4 = \chipt((w))\chipt((z)),
$$
where $w$, $z$, and $\chipt$ are as in Section \ref{sectionsplit2}. Provided \eqref{split22} is satisfied, we deduce from the equations and  definitions above that
\begin{equation}\label{xpos2}
x>0 \Longleftrightarrow pq \equiv 1\bmod 16.
\end{equation}

\section{Strategy for the Proof of the Main Theorem}\label{Strategy}
As before, let $d\in\{-4, -8, 8\}$. The ultimate goal is to prove, in the set of fundamental discriminants $D = dpq$ satisfying $\rk_4\CL(D) = 2$, a lower bound for the density of those $D$ that also satisfy $\rk_8\CL(D) \geq 1$. Suppose $p$ and $q$ are prime numbers satisfying \eqref{pqcong}. Let $w$ and $z$ be primes in $\ZZd$ satisfying \eqref{wzconditions}, and define $\alpha$ and $x$ as in \eqref{alpha}. Set $\pp = (w)$ and $\qq = (z)$. If $d = 8$, suppose that $x>0$. We define the symbol $\ve(p, q)$ by 
\begin{equation}\label{symbol}
\ve(p, q) = \left(\frac{\oalpha}{(w)}\right) = \left(\frac{\oalpha}{(z)}\right).
\end{equation}
Recall from \eqref{xpos2} that the positivity condition on $x$ can be detected via congruence conditions on $p$ and $q$ modulo $16$. Given that $p\equiv q\equiv 1\bmod 8$, there are four choices for $(p, q) \bmod 16$. When $d<0$, the positivity condition on $x$ is irrelevant, so all four of the choices are valid; however, when $d = 8$, exactly two of the choices correspond to the condition~\eqref{xpos2}. The splitting condition at the prime $\pt$ lying above $2$ can be detected via the Hecke character $\chipt$ as in \eqref{split22}. If $\chipt(\pp) = s_1$ and $\chipt(\qq) = s_2$ with $s_1, s_2\in\{\pm 1\}$, then $\alpha \equiv \square\bmod \pt^5$ if and only if $s_1s_2 = 1$.

In light of Proposition \ref{algebra}, the asymptotic formula \eqref{basicdensity1}, and the remarks above, Theorem~\ref{mainthm} is a consequence of the following theorem.
\begin{theorem}\label{Thm1}
Let $d$ be $-4$, $-8$, or $8$. Given primes $p$ and $q$ satisfying \eqref{pqcong}, let $x$ and $\alpha$ be defined as in \eqref{alpha}. Let $r_1, r_2\in\{1, 9\}$ and, in case $d = 8$, suppose that $r_1r_2 \equiv 1\bmod 16$. Let $s_1, s_2\in\{\pm 1\}$ and suppose that $s_1s_2 = 1$. Then, as $X\rightarrow\infty$, we have 
$$
\sumsum_{\substack{pq\leq X,\ p < q \\ (p, q)\equiv (r_1, r_2) \bmod 16 \\ (\chipt(\pp), \chipt(\qq)) = (s_1, s_2) \\ p\equiv\square\bmod q \\ \ve(p, q) = 1}} 1 \sim \frac{1}{1024}\frac{X\log\log X}{\log X}.
$$
\end{theorem}
Theorem \ref{Thm1} can be interpreted as follows. Classical theory of the distribution of prime numbers (see for instance \cite[Section 7.4, p.228]{MV}) gives the count of positive integers that are a product of two primes in fixed congruence classes modulo $16$, i.e., we have the asymptotic formula
\begin{equation}\label{basicdensity20}
\sumsum_{\substack{pq\leq X,\ p < q \\ (p, q)\equiv (r_1, r_2) \bmod 16}} 1 \sim \frac{1}{64}\frac{X\log\log X}{\log X}
\end{equation}
as $X\rightarrow \infty$. The conditions $\chipt(\pp) = s_1$ and $\chipt(\qq) = s_2$ can likewise be inserted without any trouble because $\chipt$ is a multiplicative character of a fixed conductor not depending on $p$ or $q$. Hence, we have
\begin{equation}\label{basicdensity2}
\sumsum_{\substack{pq\leq X,\ p < q \\ (p, q)\equiv (r_1, r_2) \bmod 16 \\ (\chipt(\pp), \chipt(\qq)) = (s_1, s_2)}} 1 \sim \frac{1}{256}\frac{X\log\log X}{\log X}
\end{equation}
as $X\rightarrow \infty$. Each of the remaining two conditions under the summation in Theorem~\ref{Thm1} can then be viewed as an event that occurs with probability one-half. Moreover, these two events are independent. To make this argument rigorous, we make use of the following formulas. Given a mathematical statement $P$, we define the indicator function of $P$ to be
$$
\ONE(P) :=
\begin{cases}
1& \text{ if }P\text{ is true} \\
0& \text{ if }P\text{ is false}.
\end{cases} 
$$
For distinct odd primes $p$ and $q$, set $\chi_p(q) := \left(\frac{p}{q}\right)$. Then we have
\begin{equation}\label{totalsumpq}
\ONE(p\equiv \square\bmod q) = \frac{1}{2}\left(1 + \chi_p(q)\right)
\end{equation}
Now we wish to generalize the character $\chipt$ to a function $\chi_2$ defined on all rational primes in a way that $\chi_2(p) = \chipt(\pp)$ for a prime $p\equiv 1\bmod 8$. We set
$$
\chi_2(p) = \frac{1}{\#\{\pp|p\}}\sum_{\pp|p}\chipt(\pp),
$$
where the sum is over all prime ideals $\pp$ in $\ZZd$ lying above $p$. With $t$ defined as in \eqref{deft}, $p$ a prime congruent to $1$ modulo $8$, and $\pp_1$ and $\pp_2$ the two primes in $\ZZd$ lying above $p$, we have
$$
\chipt(\pp_1)\chipt(\pp_2) = \left(\frac{t}{\pp_1}\right)\left(\frac{t}{\pp_2}\right) = \left(\frac{\Norm(t)}{\pp_1}\right) = \left(\frac{\Norm(t)}{p}\right) = 1,
$$
so that $\chipt(\pp_1) = \chipt(\pp_2)$. Thus indeed $\chi_2(p) = \chipt(\pp)$ whenever $p\equiv 1\bmod 8$.

Given primes $p$ and $q$, an ordered pair of integers $\br = (r_1, r_2)\in\{1, 9\}\times\{1, 9\}$, and an ordered pair of integers $\bs = (s_1, s_2)\in\{\pm 1\}\times\{\pm 1\}$, set
$$
\bc(p, q; \br, \bs) := \ONE\left((p, q)\equiv \br\bmod 16\text{ and }(\chi_2(p), \chi_2(q)) = \bs\right).
$$
Now let $p$ and $q$ be distinct primes, let $\br$ and $\bs$ be as above, and suppose that $s_1s_2 = 1$, and if $d = 8$, also that $r_1r_2\equiv 1\bmod 16$. Then we have
\begin{align*}
\ONE((p, q)\equiv \br\bmod 16,\ (\chi_2(p), \chi_2(q)) = \bs&,\text{ and }\ve(p, q) = 1) \\
& = \bc(p, q; \br, \bs)\cdot \frac{1}{2}(1+\ve(p, q)).
\end{align*}
Finally, given a vector $\be = (e_1, e_2)\in\FF_2^2$ and $p$, $q$, $\br$, and $\bs$ as above, define
\begin{equation}\label{deffpqbrbsbe}
f(p, q) = f(p, q; \br, \bs, \be) := \bc(p, q; \br, \bs)\chi_p(q)^{e_1}\ve(p, q)^{e_2}.
\end{equation}
Then, putting together the formulas above, we deduce that
$$
\sumsum_{\substack{pq\leq X,\ p < q \\ (p, q)\equiv \br \bmod 16 \\ (\chi_2(p), \chi_2(q)) = \bs \\ p\equiv\square\bmod q \\ \ve(p, q) = 1}} 1 = \frac{1}{4}\sum_{\be\in\FF_2^2}\sumsum_{\substack{pq\leq X \\ p < q }} f(p, q; \br, \bs, \be)
$$
whenever $\bs$ satisfies $s_1s_2 = 1$ and, if $d = 8$, $\br$ satisfies $r_1r_2\equiv 1\bmod 16$. If $\be = (0, 0)$, then, as we noted above in \eqref{basicdensity2}, we have
$$
\sumsum_{\substack{pq\leq X,\ p < q }} f(p, q; \br, \bs, \be) \sim \frac{1}{256}\frac{X\log\log X}{\log X}
$$
as $X\rightarrow \infty$. Hence Theorem \ref{Thm1} follows from the following oscillation statement.
\begin{theorem}\label{Thm2}
Let $\br = (r_1, r_2)\in\{1, 9\}\times\{1, 9\}$ be such that $r_1r_2\equiv 1\bmod 16$ if $d = 8$, let $\bs = (s_1, s_2)\in\{\pm 1\}\times\{\pm 1\}$ be such that $s_1s_2 = 1$, let $\be\in\FF_2^2$, and let $f(p, q; \br, \bs, \be)$ be defined as in \eqref{deffpqbrbsbe}. If $\be \neq (0, 0)$, then
$$
\sumsum_{\substack{pq\leq X,\ p < q }} f(p, q; \br, \bs, \be) = o\left(\frac{X\log\log X}{\log X}\right)
$$
as $X\rightarrow \infty$.
\end{theorem}
The rest of the paper is devoted to proving Theorem \ref{Thm2}.
 
\subsection{Summing under a hyperbola}\label{hyperbola}
We now describe how to handle sums of the form
\begin{equation}\label{SXf}
S(X; f) := \sumsum_{\substack{pq\leq X,\ p<q}}f(p, q),
\end{equation}
where $f:\ZZ\times\ZZ\rightarrow \CC$ is supported on pairs of prime numbers. The goal is to give good upper bounds for $S(X; f)$ when $f$ oscillates. Let $Y$ be a positive real number. Then
\begin{equation}\label{SABf}
S(X; f) = A(X, Y; f)+B(X, Y; f),
\end{equation}
where
\begin{equation}\label{Af}
A(X, Y; f) := \sumsum_{\substack{pq\leq X,\ p<q \\ p\leq Y}}f(p, q),
\end{equation}
and
\begin{equation}\label{Bf}
B(X, Y; f) := \sumsum_{\substack{pq\leq X\\ q>p> Y}}f(p, q).
\end{equation}
Usually $Y$ is chosen small enough compared to $X$ so that the sum $A(X, Y; f)$ can be handled using the Siegel-Walfisz theorem and variations thereof. Bounding the sum $B(X, Y; f)$ then usually proceeds by proving a double-oscillation theorem for $f$, and this type of theorem is generally useful only when $Y$ is not too small. We make these techniques precise in the following proposition.    
\begin{prop}\label{PropH}
Let $X>1$ be a real number, let $f:\ZZ\times\ZZ\rightarrow \CC$ be a function satisfying $\|f\|_{\infty}\leq 1$, and let $S(X; f)$ be defined as in \eqref{SXf}. Let $Y$ be a real number satisfying $1<Y<X^{\frac{1}{4}}$. Suppose that there exist positive real numbers $\delta_1$, $\delta_2$, and~$\delta_3$ satisfying $\delta_3<\delta_2$ such that
\begin{equation}\tag{A}
A_p(X; f) := \sum_{q\leq X}f(p, q) \ll XY^{-\delta_1}
\end{equation}
for all $p\leq Y$, where the implied constant is absolute, and such that 
\begin{equation}\tag{B}
\BB(M, N; f, \Delta) := \sumsum_{\substack{M<p\leq M+\Delta M \\ N<q\leq N+\Delta N}}f(p, q) \ll \Delta^{-\delta_2}\left(M^{-\delta_3}+N^{-\delta_3}\right)\Delta^2 MN,
\end{equation}
for all $M, N>1$ and $\Delta\in(0, 1)$ satisfying $\Delta M > M^{\frac{1}{2}}$, $\Delta N > N^{\frac{1}{2}}$, where the implied constant is absolute. Then there exists a positive real number $\delta$ in $(0, 1)$ such that
$$
S(X; f)\ll Y^{-\delta}X\log X,
$$
where the implied constant is absolute. Moreover, we can take
$$
\delta = \min\left(\frac{\delta_1}{2}, \frac{\delta_3}{2\delta_2}, \frac{\delta_3}{2}\right).
$$
\end{prop}
\begin{proof}
With $A(X, Y; f)$ defined as in \eqref{Af}, using hypothesis (A), we deduce that
\begin{equation}\label{boundAf}
\begin{array}{rcl}
A(X, Y; f) & = & \displaystyle{\sum_{p\leq Y}\left(A_p(X/p; f)-A_p(p; f)\right)} \\
& \ll & \displaystyle{\sum_{p\leq Y} \left(X p^{-1} Y^{-\delta_1} + p Y^{-\delta_1}\right)} \\
& \ll & \displaystyle{Y^{-\delta_1}X\log\log Y + Y^{2-\delta_1}} \\
& \ll & \displaystyle{XY^{-\delta_1/2}.}
\end{array}
\end{equation}
Let $\Delta = Y^{-\frac{\delta_3}{2\delta_2}}$. For each $k\geq 0$, set $M_k = N_k = Y(1+\Delta)^k$. Let $\RRR(X)$ be the region in $\RR^2$ defined by
$$
\RRR(X) = \left\{(x, y)\in \RR^2: x\geq Y, xy\leq X(1+\Delta)^{-2}, x(1+\Delta)\leq y \right\},
$$
and let $\Sigma(X)$ be the subset of $\ZZ_{\geq 0}^2$ defined by
$$
\Sigma(X) = \left\{(j, k)\in \ZZ_{\geq 0}^2: (M_j, N_k)\in\RRR(X)\right\}.
$$
If $(j , k)\in \Sigma(X)$, then the box $[M_j, M_{j+1}]\times[N_k, N_{k+1}]$ is completely contained in the region
$$
\TTT(X) = \left\{(x, y)\in \RR^2: x\geq Y, xy\leq X, x\leq y \right\}.
$$
Let $B(X, Y; f)$ be the sum defined in \eqref{Bf}. Then we can partition $B(X, Y; f)$ as
\begin{equation}\label{SBRf}
B(X, Y; f) = \sumsum_{(j, k)\in\Sigma(X)}\BB(M_j, N_k; f, \Delta) + R(X, Y; f, \Delta).
\end{equation}
As $\|f\|_{\infty}\leq 1$, we give a trivial upper bound for $R(X, Y; f, \Delta)$ by counting lattice points in the region $\TTT(X)\setminus\RRR(X)$, i.e.,
$$
\left|R(X, Y; f, \Delta)\right| \leq \#\left(\ZZ^2\cap\left(\TTT(X)\setminus\RRR(X)\right)\right).
$$
The right-hand side above can be approximated by the area of the region $\TTT(X)\setminus\RRR(X)$, with an error term bounded by the sum of the lengths of the projections of $\TTT(X)\setminus\RRR(X)$ to the axes (this is known as the Lipschitz principle; see \cite{Dav} and \cite{DavC}). Thus we have
\begin{equation}\label{boundRf}
\begin{array}{rcl}
R(X, Y; f, \Delta) & \ll & \displaystyle{\int_{0}^{X^{\frac{1}{2}}}\Delta x dx + \int_{Y}^{X^{\frac{1}{2}}}\frac{\left(X-X/(1+\Delta)^2\right)}{x}dx + X^{\frac{1}{2}} + XY^{-1} + 1}\\
& \ll & \displaystyle{\Delta X + X\frac{2\Delta+\Delta^2}{(1+\Delta)^2}\log\left(\frac{X^{\frac{1}{2}}}{Y}\right)+ X^{\frac{1}{2}} + XY^{-1} + 1} \\
& \ll & \displaystyle{\Delta X + \Delta X\log X + XY^{-1}} \\
& \ll & \displaystyle{Y^{-\frac{\delta_3}{2\delta_2}}X\log X + Y^{-1}X}.
\end{array} 
\end{equation}
\\\\
As $\delta_3<\delta_2$ and $M_j, N_k\geq Y$, we have $\Delta M_j> M_j^{\frac{1}{2}}, \Delta N_k > N_k^{\frac{1}{2}}$. Thus we can use hypothesis (B) to give the bound
\begin{equation}\label{boundSBf}
\begin{array}{rcl}
\displaystyle{\sumsum_{(j, k)\in\Sigma(X)}\BB(M_j, N_k; f, \Delta)} & \ll & \displaystyle{\Delta^{-\delta_2}Y^{-\delta_3}\sumsum_{(j, k)\in\Sigma(X)}\Delta^2M_jN_k} \\
& \ll & \displaystyle{Y^{-\frac{\delta_3}{2}}\cdot \Area\TTT(X)}\\
& \ll & \displaystyle{Y^{-\frac{\delta_3}{2}}X\log X}.
\end{array}
\end{equation}
Combining \eqref{boundAf}, \eqref{boundRf}, and \eqref{boundSBf}, we deduce the proposition.
\end{proof}

To apply Proposition \ref{PropH}, we will prove the following two propositions. In the following, define $f(p, q; \br, \bs, \be)$ as in Theorem \ref{Thm2}, and suppose $\be\neq(0, 0)$.
\begin{prop}\label{PropA}
Let $f(p, q) = f(p, q; \br, \bs, \be)$. Then there is a constant $c>0$ such that for all $p\leq \left(\log X\right)^{100}$, we have
$$
A_p(X; f) = \sum_{q\leq X}f(p, q) \ll X\exp\left(-c(\log X)^{\frac{1}{4}}\right),
$$
where the implied constant is absolute (but ineffective).
\end{prop}
\begin{prop}\label{PropB}
Let $f(p, q) = f(p, q; \br, \bs, \be)$. Then, for all $M, N > 1$ and $\Delta\in(0, 1)$ satisfying $\Delta M, \Delta N > 1$, we have
$$
\BB(M, N; f, \Delta) = \sumsum_{\substack{M<p\leq M+\Delta M \\ N<q\leq N+\Delta N}}f(p, q) \ll \Delta^{-\frac{11}{12}}\left(M^{-\frac{1}{12}}+N^{-\frac{1}{12}}\right)\Delta^2 MN,
$$
where the implied constant is absolute.
\end{prop}
Hence, assuming Propositions \ref{PropA} and \ref{PropB}, we can apply Proposition \ref{PropH} with $Y = \left(\log X\right)^{100}$, $\delta_1 = 1$, $\delta_2 = \frac{11}{12}$, and $\delta_3 = \frac{1}{12}$ to obtain Theorem \ref{Thm2}. Our goal is now to prove Propositions \ref{PropA} and \ref{PropB}.

\section{Preliminaries on Quadratic Characters in Quadratic Fields}\label{Preliminaries}
In this section, we prove some properties of quadratic residue symbols in quadratic number fields. The most important among them is Proposition~\ref{keycancellation} below, which is a generalization of \cite[Lemma 21.1, p.\ 1025]{FI1}. We have made an effort to make our proof more conceptually clear. Throughout this section, let $L$ denote the quadratic number field of discriminant $D$, and let $\OO_L = \ZZ[(D+\sqrt{D})/2]$ denote its maximal order.

\subsection{Primitivity}
We say that an ideal $\aaa$ in $\OO_L$ is \textit{primitive} if $\gcd(\aaa, \oaaa) = 1$. Furthermore, we say that an element $w\in\OO_L$ is \textit{primitive} if the principal ideal generated by $w$ is primitive. If $\aaa$ is primitive, then all prime ideals dividing $\aaa$ are unramified of degree one, and so the inclusion $\ZZ\hookrightarrow \OO_L$ induces an isomorphism
\begin{equation}\label{primitive1}
\ZZ/(\Norm(\aaa))\stackrel{\sim}{\longrightarrow} \OO_L/\aaa. 
\end{equation}
We call an ideal $\aaa$ (resp.\ element $w$) in $\OO_L$ \textit{odd} if $\Norm(\aaa)$ (resp.\ $\Norm(w)$) is an odd integer.
\begin{remark}
For instance, in $\ZZI$, the principal ideal $(5)$ is odd but not primitive. Indeed, note that $\Norm(5) = 25$, but $\ZZI/(5)\cong \ZZI/(2+\sqrt{-1})\times \ZZI/(2-\sqrt{-1})\cong \ZZ/(5)\times\ZZ/(5) \ncong \ZZ/(25)$.
\end{remark}
If $\aaa$ is primitive and odd, then for every rational integer $n$ we have the equality of quadratic residue symbols
\begin{equation}\label{primitive2}
\left(\frac{n}{\Norm(\aaa)}\right) = \left(\frac{n}{\aaa}\right),
\end{equation}
where the symbol on the left is the usual Jacobi symbol while the symbol on the right is the quadratic residue symbol in $\OO_L$. By \eqref{primitive1} and \eqref{primitive2}, we deduce that 
\begin{equation}\label{primitive3}
\sum_{z\in \OO_L/\aaa}\left(\frac{z}{\aaa}\right) = \sum_{n\in \ZZ/(\Norm(\aaa))}\left(\frac{n}{\Norm(\aaa)}\right).
\end{equation}
Suppose $\aaa$ and $\bb$ are ideals in $\OO_L$. If either of $\aaa$ and $\bb$ is not primitive, then their product $\aaa\bb$ is not primitive. Even if $\aaa$ and $\bb$ are both primitive, we will now see that their product $\aaa\bb$ is \textit{not} necessarily primitive.
\begin{lemma}\label{primitive5}
Suppose $\aaa$ and $\bb$ are primitive. Let $\rrr = \gcd(\aaa, \obb)$ and $r = \Norm(\rrr)$. Then $\aaa\bb/(r)$ is primitive. In particular, $\aaa\bb$ is primitive if and only if $\gcd(\aaa, \obb) = (1)$.
\end{lemma}
\begin{proof}
Write $\aaa = \rrr\aaa_1$ and $\bb = \overline{\rrr}\bb_1$ with $\gcd(\aaa_1, \obb_1) = (1)$. The claim then is that $\aaa_1\bb_1$ is primitive. As $\aaa$ is primitive, we have $\gcd(\aaa_1, \oaaa_1) = (1)$ and hence $\gcd(\aaa_1, \overline{\aaa_1\bb_1}) = (1)$. Similarly, as $\bb$ is primitive, we deduce that $\gcd(\bb_1, \overline{\aaa_1\bb_1}) = (1)$, and thus the claim is proved.
\end{proof}
Given two primitive ideals $\aaa$ and $\bb$ and $r$ as above, we now show that we can obtain a primitive ideal by dividing $\aaa\bb$ by an ideal of norm $r$ (as opposed to $r^2$). As before, we set $\rrr = \gcd(\aaa, \obb)$ and we write $\rrr = \rrr_a\rrr_b$, where
$$
\rrr_a = \prod_{\substack{\pp^k\|\rrr \\ \val_{\pp}(\aaa)<\val_{\pp}(\obb)}}\pp^k, \ \ \ \ \ \ \text{ and }\ \ \ \ \ \ \rrr_b = \prod_{\substack{\pp^k\|\rrr \\ \val_{\pp}(\aaa)\geq \val_{\pp}(\obb)}}\pp^k.
$$
We set $\ccc = \rrr_a\overline{\rrr_b}$. Then clearly $\Norm(\ccc) = \Norm(\rrr) = r$. Moreover, by construction 
$$
\gcd\left(\frac{\aaa}{\rrr_{a}}, \frac{\obb}{\rrr_{b}}\right) = (1),
$$
so by Lemma \ref{primitive5}, we conclude $\aaa\bb/\ccc$ is primitive. As $\gcd(\rrr_a, \rrr_b) = (1)$, we see that $\ccc$ is also primitive. Finally, we claim that $\gcd(\ccc, \aaa\bb/\ccc) = (1)$. Indeed, suppose that $\pp$ divides $\ccc$. First, if $\pp$ divides $\rrr_a$, then $\pp$ doesn't divide $\aaa/\rrr_a$ by construction and $\pp$ doesn't divide $\bb/\overline{\rrr_b}$ because $\bb$ is primitive. Similarly, if $\pp$ divides $\overline{\rrr_b}$, then $\pp$ doesn't divide $\bb/\overline{\rrr_b}$ by construction and $\pp$ doesn't divide $\aaa/\rrr_a$ because $\aaa$ is primitive. This proves the claim. Now the Chinese Remainder Theorem and \eqref{primitive1} imply that
\begin{equation}\label{primitive6}
\OO_L/\aaa\bb \cong \OO_L/(\aaa\bb/\ccc)\times\OO_L/\ccc \cong \ZZ/(Y/r)\times \ZZ/(r),
\end{equation}
where $Y = \Norm(\aaa\bb)$.

\subsection{Cancellation in quadratic character sums}
The rough idea behind proving that the symbol $\ve(p, q)$ defined in \eqref{symbol} oscillates as $p$ and $q$ vary in a box where neither $p$ nor $q$ is too small is to give meaning to $\ve(m, n)$ for all integers $m$ and $n$, then to prove that the bilinear sum 
$$
\sumsum_{m, n}a_mb_n\ve(m, n)
$$
oscillates for any bounded sequences $\{a_m\}_m$ and $\{b_n\}_n$, and finally to apply this result to sequences $\{a_m\}_m$ and $\{a_n\}_n$ supported on the primes. The following definition generalizes the symbol $\ve(p, q)$ in a way that will allow us to apply this method. 

Let $w, z\in \OO_L$ and suppose that $w$ is odd. We define the quadratic multiplicative character $\chiw : \OO_L\rightarrow \{-1, 0, 1\}$ by setting
$$
\chiw(z) := \left(\frac{z}{(w)}\right),
$$
and we define the multiplier factor $\mul(w)$ for odd elements $w\in\OO_L$ by setting $\mul(w):=\chiw(\ow)$. We note that $\chiw(z) \neq 0$ if and only if $\gcd((z), (w)) = (1)$, and so in particular Lemma~\ref{primitive5} implies that $\mul(w)$ is supported on primitive elements $w$. \begin{remark}
Suppose $w\in\OO_L$ generates an odd prime of degree $1$, so that $\Norm(w)=p$ for some rational prime $p$. Then
\begin{equation}\label{linkjacgamma}
\chiw(z)\chiw(\oz)  = \left(\frac{\Norm(z)}{(w)}\right)  = \left(\frac{\Norm(z)}{p}\right),
\end{equation}
where the symbol on the far right is the usual Jacobi symbol.
\end{remark}
\begin{remark}
In the particular case when $w$ and $z$ are primes in $\ZZd$ (with $d_0\in\{-1, -2, 2\}$) lying above rational primes $p$ and $q$, respectively, satisfying \eqref{pqcong}, \eqref{pqjac}, \eqref{wzconditions}, \eqref{split2}, and, if $d = 8$, also \eqref{xpos}, then $\ve(p, q)$ as defined in \eqref{symbol} can be written as
\begin{equation}\label{linkvegamma}
\ve(p, q) = \mul(w)\chiw(\oz).
\end{equation}
\end{remark}
\begin{remark}
The Dirichlet symbol defined in \cite[Equation (19.11), p.\ 1019]{FI1} is simply equal to $\mul(w)\chiw(\oz)$ in the special case that $L = \QQ(\sqrt{-1})$.
\end{remark}
The following proposition provides all of the cancellation that we need for Proposition \ref{PropB}.
\begin{prop}\label{keycancellation}
Let $w_1, w_2\in\OO_L$ be odd and primitive. Let $\rrr = \gcd((w_1), (\ow_2))$, $r = \Norm(r)$, $Y = \Norm(w_1w_2)$. Then
$$
\left|\sum_{\substack{z\in \OO_L/(Y)}}\chi_{w_1}(z)\chi_{w_2}(z) \right| =
\begin{cases}
Y\varphi(r)\varphi(Y/r) & \text{if }Y\text{ and }r\text{ are squares} \\
0 & \text{otherwise.}
\end{cases}
$$
\end{prop}
\begin{proof}
We have
$$
\sum_{\substack{z\in \OO_L/(Y)}}\chi_{w_1}(z)\chi_{w_2}(z) = \sum_{\substack{z\in \OO_L/(Y)}} \left(\frac{z}{(w_1w_2)}\right) = Y\sum_{\substack{z\in \OO_L/(w_1w_2)}}\left(\frac{z}{(w_1w_2)}\right).
$$
By \eqref{primitive6}, we have $\OO_L/(w_1w_2) \cong \OO_L/\ccc_1\times \OO_L/\ccc_2$, where $\ccc_1$ and $\ccc_2$ are coprime primitive ideals of norm $Y/r$ and $r$, respectively, satisfying $(w_1w_2) = \ccc_1\ccc_2$. Hence
$$
\sum_{\substack{z\in \OO_L/(w_1w_2)}}\left(\frac{z}{(w_1w_2)}\right) = \sumsum_{\substack{z_{1}\bmod \ccc_1 \\ z_{2}\bmod \ccc_2}}\left(\frac{z'}{\ccc_1\ccc_2}\right),
$$
where $z' = z'(z_{1}, z_{2})$ is the unique congruence class modulo $\ccc_1\ccc_2$ such that $z'\equiv z_{1}\bmod \ccc_1$ and $z' \equiv z_{2}\bmod \ccc_2$. With these choices, we have
$$
\left(\frac{z'}{\ccc_1\ccc_2}\right) = \left(\frac{z'}{\ccc_1}\right)\left(\frac{z'}{\ccc_2}\right) = \left(\frac{z_{1}}{\ccc_1}\right)\left(\frac{z_{2}}{\ccc_2}\right). 
$$ 
Then, by \eqref{primitive3}, we have 
$$
\sum_{z_{1}\bmod \ccc_1}\left(\frac{z_{1}}{\ccc_1}\right) \sum_{z_{2}\bmod \ccc_2}\left(\frac{z_{2}}{\ccc_2}\right) = \sum_{c_1\in \ZZ/(Y/r)}\left(\frac{c_1}{Y/r}\right)\sum_{c_2\in \ZZ/(r)}\left(\frac{c_2}{r}\right),
$$
where the symbols on the right-hand side of the equality are the usual Jacobi symbols. For any positive integer $n$, we have
$$
\sum_{a\in \ZZ/(n)}\left(\frac{a}{n}\right) = 
\begin{cases}
\varphi(n)&\text{ if }n\text{ is a square}, \\
0&\text{ otherwise}.
\end{cases}
$$
Combining all of the equations above, we conclude the proof of the proposition.
\end{proof}

\subsection{A family of Hecke characters for $\ZZd$}\label{HeckeFamily}\label{defpsi}
We now return to the case that $d\in\{-4, -8, 8\}$, and we set $d_0  = d/4$, as before. The function $\chiw$ is a character on $\left(\ZZd/(w)\right)^{\times}$. We now show that this character can be completed into a Hecke character $\psi_w$ for $\QQd$ in the case that $w$ is a prime of degree $1$ in $\ZZd$ satisfying $\Norm(w) = p \equiv 1\bmod 8$. We must define a homomorphism $\psi_w$ on the group $\III(w)$ of fractional ideals of $\ZZd$ coprime to $(w)$, i.e., 
$$
\psi_w: \III(w) \rightarrow S^{1} = \{s\in \CC:|s| = 1\},
$$
such that there exists a continuous function
$$
\chi_{w, \infty}: F^{\times}\rightarrow S^{1}
$$
satisfying $\chiw(u)\chi_{w, \infty}(u) = \psi_w((u)) = 1$ for all units $u\in\ZZd^{\times}$; here
$$
F^{\times} =
\begin{cases}
\CC^{\times} & \text{ if }d = -4\text{ or } -8, \\
\RR^{\times}\times\RR^{\times} & \text{ if }d_0 = 8. 
\end{cases}
$$
\subsubsection{The cases $d = -4$ and $d = -8$}
For the case $d = -8$, note that $\chiw$ is trivial on $\ZZJ^{\times} = \{\pm 1\}$ because $p\equiv 1\bmod 4$ (i.e., $\left(\frac{-1}{p}\right) = 1$). For the case $d = -4$, suppose $w = a+b\sqrt{-1}$ with $a, b\in\ZZ$, $a$ odd, and $b = 2^kb'$ with $b'$ odd. As $\gcd(b, p) = 1$, we can write $\sqrt{-1}\equiv -a/b \bmod w$. As $p = a^2+b^2 \equiv 1\bmod 8$, we have
$$
\begin{array}{rcl}
\chiw(\sqrt{-1}) & = & \displaystyle{\left(\frac{\sqrt{-1}}{(w)}\right) = \left(\frac{-a/b}{(w)}\right) = \left(\frac{-ab}{(w)}\right)} \\
 & = & \displaystyle{\left(\frac{-ab}{p}\right) = \left(\frac{|a|}{p}\right)\left(\frac{|b'|}{p}\right)} \\
 & = & \displaystyle{\left(\frac{p}{|a|}\right) \left(\frac{p}{|b'|}\right) = 1\cdot 1 = 1},
\end{array}
$$
where the symbols from the second line onwards are usual Jacobi symbols. Hence $\chiw$ is trivial on $\ZZI^{\times}$. Therefore, in case $d \in\{-4, -8\}$, we can extend $\chiw$ to a character on ideals in $\ZZd$ by setting $\chiw(\aaa) := \chiw(z)$, where $z$ is any generator of $\aaa$. Now it suffices to take $\chi_{w, \infty}$ to be identically $1$ on all of $\CC^{\times}$. Then setting $\psi_w(\aaa) := \chiw(\aaa)$ defines is a character on $\III(w)$. Moreover, by \eqref{linkvegamma}, if $p$ and $q$ are primes satisfying \eqref{pqcong} and \eqref{pqjac}, and $w$ and $z$ are primes in $\ZZd$ satisfying \eqref{wzconditions} and \eqref{split2}, then we have
$\ve(p, q) = \mul(w)\psi_{w}((\oz))$.

\subsubsection{The case $d = 8$}
Suppose $w = a+b\sqrt{2}$ with $a, b\in\ZZ$. Then, as $p \equiv 1\bmod 8$, $b$ must be even. The unit group $\ZZT^{\times}$ is generated by $-1$ and $\ve = 1+\sqrt{2}$. We have $\chiw(-1) = \left(\frac{-1}{p}\right) = 1$. If we write $b = 2^kb'$ with $b'$ odd, we have
\begin{equation}\label{chiwve}
\begin{array}{rcl}
\chiw(\ve) & = & \displaystyle{\left(\frac{1+\sqrt{2}}{(w)}\right) = \left(\frac{1-a/b}{(w)}\right)} \\
 & = & \displaystyle{\left(\frac{1-a/b}{p}\right) = \left(\frac{b^2-ab}{p}\right) = \left(\frac{b}{p}\right)\left(\frac{a-b}{p}\right)} \\
 & = & \displaystyle{\left(\frac{p}{|b'|}\right) \left(\frac{p}{|a-b|}\right) = 1\cdot \left(\frac{p-(a^2-b^2)}{|a-b|}\right) = \left(\frac{-1}{|a-b|}\right)},
\end{array}
\end{equation}
where again the symbols from the second line onwards are usual Jacobi symbols. Every other generator for the ideal $(w)$ of norm $p$ is of the form $\pm\ve^{2k}w$, where $k$ is an integer. As $\ve^2(a+b\sqrt{2}) = (3a+4b)+(2a+3b)\sqrt{2}$ and $(3a+4b) - (2a+3b) = a + b \equiv a - b \bmod 4$, the last line of \eqref{chiwve} implies that $\chi_w(\ve) = \chi_{\ve^2w}(\ve)$. Moreover, again by the last line of \eqref{chiwve}, we have $\chi_{-w}(\ve) = \left(\frac{-1}{|-a+b|}\right) = \left(\frac{-1}{|a-b|}\right) = \chi_w(\ve)$. Thus we cannot always choose a generator $w$ of a prime ideal lying above $p$ satisfying both $\Norm(w) = p$ and $\chi_w(\ve) = 1$. In fact, we have 
$$
\chiw(\ve) = 
\begin{cases}
1&\text{ if }|a-b|\equiv 1\bmod 4, \\
-1&\text{ otherwise}.
\end{cases}
$$
We will define a different Hecke character $\psi_w$ modulo $(w)\infty_1\infty_2$ in each of the cases above; here $\infty_1$ and $\infty_2$ are the two embeddings $\QQ(\sqrt{2})\hookrightarrow\RR$. If $\chi_w(\ve) = 1$, then $\chi_w$ is already a character on fractional ideals in $\ZZT$ and we simply define $\psi_w: \III(w)\rightarrow S^1$ by setting $\psi_w(\aaa) := \chiw(z)$, where $z$ is any generator of $\aaa$. In this case, we again take $\chi_{w, \infty}$ to be identically $1$ on all of $\RR^{\times}\times \RR^{\times}$. If $\chiw(\ve) = -1$, we take $\chi_{w, \infty}(z) = \sign(\Norm(z))$, and define $\psi_w(\aaa) = \chiw(z)\chi_{w, \infty}(z)$, where $z$ is any generator of $\aaa$. The homomorphism $\psi_w$ is well-defined because $\chiw(\ve)\chi_{w, \infty}(\ve) = -1\cdot -1 = 1$ and $\chiw(-1)\chi_{w, \infty}(-1) = 1\cdot 1 = 1$.

We note that in both cases, if $z\equiv 1\bmod^{\times}(w)\infty_1\infty_2$, so that $z, \oz>0$, then $\psi_w((z)) = 1 = \sign(z)$. Furthermore, similarly as in the cases $d = -4$ and $d = -8$, if $p$ and $q$ are primes satisfying \eqref{pqcong} and \eqref{pqjac} and $w$ and $z$ are primes in $\ZZT$ satisfying \eqref{wzconditions}, \eqref{split2}, \eqref{xpos}, then we have $\ve(p, q) = \mul(w)\psi_{w}((\oz))$.

Finally, we remark that if $w$ and $z$ in $\ZZd$ (\textit{any} $d_0\in\{-1, -2, 2\}$) satisfying \eqref{wzconditions}, \eqref{split2}, and also \eqref{xpos} if $d_0 = 2$, then so do $w$ and $\oz$. Hence
\begin{equation}\label{vepsi}
\ve(p, q) = \mul(w)\psi_{w}(\qq),
\end{equation}
where $\qq$ is any prime ideal in $\ZZd$ lying above $q$.

\section{Proof of Proposition \ref{PropA}}\label{PropAproof}
In this section, we exploit the arithmetic of $\QQd$ ($d\in\{-4, -8, 8\}$) to prove that $\ve(p, q)$ oscillates when $q$ varies over a range much bigger than the size of $p$. The main tool is the theory of Hecke $L$-functions. Let us first recall the sum from Proposition \ref{PropA}. We let
$$
A_p(X; f) = \sum_{q\leq X}f(p, q; \br, \bs, \be),
$$
where $\br = (r_1, r_2)\in \{1, 9\}\times \{1, 9\}$, $r_1r_2\equiv 1\bmod 16$ if $d = 8$, $\bs = (s_1, s_2)\in\{\pm 1\}\times \{\pm 1\}$, $s_1s_2 = 1$, $\be\in\FF_2^2$, $\be\neq (0, 0)$, and 
$$
f(p, q; \br, \bs, \be) = 
\begin{cases}
\chi_p(q)^{e_1}\ve(p, q)^{e_2} & \text{if }(p, q)\equiv \br\bmod 16 \text{ and }(\chi_2(p), \chi_2(q)) = \bs\\
0 & \text{otherwise.} 
\end{cases}
$$
Hence $A_p(X; f)$ vanishes unless $p\equiv r_1\bmod 16$ and $\chi_2(p) = s_1$. So let $p$ be a prime number satisfying $p\equiv r_1\bmod 16$ and $\chi_2(p) = s_1$, let $w \in \ZZd$ be a prime satisfying \eqref{wzconditions}, and let $\psi_w$ be the Hecke character on $\QQd$ defined in Section~\ref{defpsi}. By \eqref{vepsi}, we have
$$
f(p, q; \br, \bs, \be) =
\begin{cases}
\left(\frac{p}{\qq}\right)^{e_1}(\mul(w)\psi_w(\qq))^{e_2} & \text{if }q\equiv r_2\bmod 16\text{ and }\chi_2(q) = s_2 \\
0 & \text{otherwise,}
\end{cases}
$$
where $\qq$ is a prime ideal in $\ZZd$ dividing~$q$. To use the theory of $L$-functions for the number field $\QQd$, we now define a function on \textit{all} ideals $\qq$ in $\ZZd$. Let
$$
f_1(\qq; w, \be) := \left(\frac{p}{\qq}\right)^{e_1}\psi_w(\qq)^{e_2}.
$$
We can detect the congruence condition $q\equiv r_2\bmod 16$ via Dirichlet characters modulo $16$ and the condition $\chi_2(q) = s_2$ via the formula
$$
\frac{1}{2}\left(1+s_2\chi_2(q)\right) = 
\begin{cases}
1&\text{if }\chi_2(q) = s_2 \\
0&\text{otherwise.}
\end{cases}
$$
Then, with $c_p = 32 \cdot \mul(w)^{e_2}$, we have 
\begin{align*}
c_p\cdot A_p(X; f) & = \sumsum_{\substack{\chi_{16}\bmod 16\\ e_3\in\FF_2}}\sum_{\substack{\qq\text{ split} \\ \Norm(\qq)\leq X}} \chi_{16}(r_2\Norm(\qq))(s_2\chipt(\qq))^{e_3}f_1(\qq; w, \be) \\
 & = \sumsum_{\substack{\chi_{16}\bmod 16\\ e_3\in\FF_2}}\sum_{\substack{\qq \\ \Norm(\qq)\leq X}} \chi_{16}(r_2\Norm(\qq))(s_2\chipt(\qq))^{e_3}f_1(\qq; w, \be) \\
 & \ \ \ \ -\sumsum_{\substack{\chi_{16}\bmod 16\\ e_3\in\FF_2}}\sum_{\substack{\qq\text{ inert} \\ \Norm(\qq)\leq X}} \chi_{16}(r_2\Norm(\qq))(s_2\chipt(\qq))^{e_3}f_1(\qq; w, \be),
\end{align*}
where the outer sums are over Dirichlet characters $\chi_{16}$ modulo $16$ and elements $e_3\in\FF_2$. But if a prime ideal $\qq = (q)$ in $\ZZd$ is inert, then $\Norm(\qq) = q^2$, so
$$
\sum_{\substack{\qq\text{ inert} \\ \Norm(\qq)\leq X}} 1 \ll X^{\frac{1}{2}}.
$$
Hence, to prove Proposition \ref{PropA}, it remains to show, for each Dirichlet character $\chi_{16}$ and element $e_3\in\FF_2$, that there exists a constant $c>0$ such that 
$$
\sum_{\Norm(\qq)\leq X} \chi_{16}(\Norm(\qq))\chipt(\qq)^{e_3}f_1(\qq; w, \be) \ll X\exp\left(c\sqrt{\log X}\right)
$$
for all $p = \Norm(w)\leq \left(\log X\right)^{100}$. We now apply the theory of Hecke $L$-functions to obtain this bound. Define the Hecke character $\psi$ for $\ZZd$ by setting
$$
\psi(\qq) = \chi_{16}(\Norm(\qq))\chipt(\qq)^{e_3}f_1(\qq; w, \be).
$$
We claim that the function $\qq\mapsto \psi(\qq)$ is a non-trivial Hecke character for $\ZZd$ of conductor $\ff$ satisfying $\Norm(\ff)\ll p^2$, where the implied constant is absolute. First, note that $\qq\mapsto \chi_{16}(\Norm(\qq))\chipt(\qq)^{e_3}$ is a Hecke character of conductor dividing a power of $2$. If $e_1 = 1$ and $e_2 = 0$, then the claim follows because $\qq\mapsto \left(\frac{p}{\qq}\right)$ is a non-trivial Hecke character of conductor $(p)$. If $e_1 = 0$ and $e_2 = 1$, then the claim follows because $\qq\mapsto \psi_w(\qq)$ is a non-trivial Hecke character of conductor $(w)$, as shown in Section \ref{HeckeFamily}. Finally, if $e_1 = e_2 = 1$, then by \eqref{linkjacgamma}, we have
$$
\left(\frac{p}{\qq}\right)\psi_w(\qq) = \psi_w(\overline{\qq}) = \psi_{\ow}(\qq),
$$
so that $\qq\mapsto \left(\frac{p}{\qq}\right)\psi_w(\qq)$ is a non-trivial Hecke character of conductor $(\ow)$.

Now that we have established the claim, we use a version of the Siegel-Walfisz Theorem for Hecke $L$-functions. As usual, we define the Hecke $L$-function
$$
L(s, \psi) = \sum_{\aaa}\psi(\aaa)\Norm(\aaa)^{-s} \ \ \ (\Re(s)>1),
$$
where the sum is over all non-zero ideals $\aaa\subset \ZZd$. By \cite[Theorem 3.3.1, p. 93]{Miy}, $L(s, \psi)$ has a meromorphic continuation to $\CC$ and satisfies a functional equation as well as other standard properties of $L$-functions. As $\psi$ is not the trivial character, the order of the pole at $s = 1$ of $L(s, \psi)$ is $0$. Hence \cite[Main Theorem, p. 418]{Gold} (with, say, $\varepsilon = \frac{1}{800}$) implies that there is an absolute constant $c>0$ such that for all $p\leq \left(\log X\right)^{100}$, we have
$$
\sum_{\Norm(\qq)} \psi(\qq) \ll X\exp\left(-c(\log X)^{\frac{1}{4}}\right).
$$
This completes the proof of Proposition \ref{PropA}.
\begin{remark}
The range of $p$ for which the above bound holds could be extended to $\exp\left(c'\sqrt{\log X}\right)$ for some small $c'>0$ instead of a power of $\log X$ if we were certain that $L(s, \psi)$ has no Siegel zeros. Although this is conjectured to be true in any case, we can only show it in the cases when $d \in\{-4, -8\}$ and $e_2 = 1$. In all cases $d \in\{-4, -8, 8\}$, when $e_2 = 1$, the theta series 
$$
\Theta(z, \psi) = \sum_{\aaa}\psi(\aaa)\exp(2\pi i \Norm(\aaa))
$$
is a holomorphic modular form of weight $1$ and level $4p$ (see \cite[Theorem 4.8.2, p. 183]{Miy}). Now a theorem of Hoffstein and Ramakrishnan \cite[Theorem C, p.299]{HofRam} implies that the associated $L$-function $L(s, \psi)$ has no Siegel zero whenever $\Theta(z, \psi)$ is a cusp form. If $d = -4$ or $d= -8$, then this is indeed the case. Otherwise, if $d = 8$, unfortunately $\Theta(z, \psi)$ is \textit{not} a cusp form.
\end{remark}

\section{Proof of Proposition \ref{PropB}}
In this section we finish the proof of Proposition \ref{PropB} and hence also of Theorem~\ref{mainthm}. We will use power-saving upper bounds for very general bilinear sums that were obtained in \cite{FI1} for $d = -4$ and \cite{Milovic2} for $d = 8$. We first prove an estimate that holds in general quadratic fields. 

\subsection{Double-oscillation estimates for $\chiw(z)$}\label{Doubleoscillation}
Let $L$ be the quadratic number field of discriminant $D$, and let $\OO_L$ be its ring of integers, as before. For a subset $S\subset \RR^2$ and an element $w = a+b\sqrt{D}\in L$ with $a, b\in\QQ$, we will write $w\in S$ if and only if $(a, b)\in S$. If $D>0$, we define a region $\RRR_k \subset \RR^2$ for each integer $k\geq 1$ by setting
$$
\RRR_k := \left\{(x, y)\in\RR^2:\ x>0, |y|\leq D^{-\frac{1}{2}}\frac{\ve_D^k-1}{\ve_D^k+1}x\right\}.
$$
where $\ve_D$ is the fundamental unit of norm $+1$, i.e., the smallest element of $\OO_L\hookrightarrow \RR$ satisfying $\ve_D>1$ and $\Norm(\ve_D) = 1$. Then $\RRR_k$ looks like a cone emanating from the origin. Moreover, the set of $\alpha\in \OO_L\cap\RRR_k$ is exactly the set of totally positive $\alpha\in\OO_L$ satisfying $\ve_D^{-k}\leq \overline{\alpha}/\alpha \leq \ve_D^k$. Hence every principal ideal of $\OO_L$ that can be generated by a totally positive element has exactly one generator in $\RRR_1$ (see for instance \cite[Chapter 6]{Marcus}). If $D<0$, we simply set $\RRR_k = \RR^2$. Finally, given positive real numbers $X$ and $\Delta$, we set
$$
\RRR_k(X; \Delta) := \left\{(x, y)\in \RRR_k:\ X\leq \Norm(x+y\sqrt{D}) < X(1+\Delta)\right\}.
$$
We note that there are constants $c_1 = c_1(D, k) > 0$ and $c_2 = c_2(D, k) > 0$ such that the $2$-dimensional volume of $\RRR_k(X; \Delta)$ is bounded by $c_2 \Delta X$ and such that the sum of the $1$-dimensional volumes of the projections of $\RRR_k(X; \Delta)$ on the coordinate axes is bounded by $c_1 X^{\frac{1}{2}}$.

We now define the general bilinear sum of interest. Given two sequences of complex numbers $\alpha = \{\alpha_{w}\}$ and $\beta = \{\beta_{z}\}$ indexed by elements of $\OO_L$, and real numbers $W, Z > 0$, we set
$$
B(W, Z; \alpha, \beta):= \suma_{\substack{w \in\RRR_1(W; \Delta)}}\suma_{\substack{z \in \RRR_1(Z; \Delta)}}\alpha_{w}\beta_{z}\chiw(z),
$$
where $\ast$ restricts the sums to primitive elements. We will prove
\begin{prop}\label{doubleosc}
There exists a real number $C>0$ such that: for all real numbers $W, Z > 1$, for all real numbers $\Delta \in (0, 1)$ satisfying $\Delta W > W^{\frac{1}{2}}$, $\Delta Z \geq Z^{\frac{1}{2}}$, and for all sequences of complex numbers $\alpha = \{\alpha_{w}\}$ and $\beta = \{\beta_{z}\}$ satisfying $|\alpha_w|, |\beta_z|\leq 1$ and supported on $w$ and $z$ such that the principal ideals $(w)$ and $(z)$ each have at most $f$ prime ideal factors in $\OO_L$, we have 
$$
\left|B(W, Z; \alpha, \beta)\right| \leq C\cdot 2^{40f} \Delta^{-\frac{11}{12}}\left(W^{-\frac{1}{12}}+Z^{-\frac{1}{12}}\right)\Delta^2 WZ\log(W+Z).
$$
\end{prop}
\begin{proof}
Fix an integer $k\geq 1$. The implied constants in the $\ll$ symbols that follow may depend on $k$ in some cases, but we suppress this dependence since we will ultimately take $k = 3$. By H\"{o}lder's inequality (with $\frac{2k-1}{2k}+\frac{1}{2k} = 1$), we have
\begin{equation}\label{bbound1}
\left|B(W, Z; \alpha, \beta)\right|^{4k} \leq \left(\suma_{w}|\alpha_w|^{\frac{2k}{2k-1}}\right)^{4k-2}\cdot \left(\suma_{w}\left|\suma_{z}\beta_z\chiw(z)\right|^{2k}\right)^2,
\end{equation}
where the sums over $w$ and $z$ are implicitly restricted to $w\in\RRR_1(W, \Delta)$ and $z\in\RRR_1(Z, \Delta)$, each having at most $f$ prime ideal factors. By the Lipschitz principle (see~\cite{Dav}), since $|\alpha_w| \leq 1$, and since $1\leq W^{\frac{1}{2}}\leq \Delta W$, the first factor on the right-hand side of \eqref{bbound1} is
\begin{equation}\label{bbound2}
\ll (\Delta W)^{4k-2}.
\end{equation}
We expand the inner sum in the second factor on the right-hand side of \eqref{bbound1} to get
$$
\left|\suma_{z}\beta_z\chiw(z)\right|^{2k} = \sum_{z}\beta_z'\chiw(z),
$$
where
$$
\beta_z' = \sum_{\substack{z = z_1\cdots z_{2k} \\ z_1,\ldots, z_{2k}\in \RRR_1(Z, \Delta) \\ z_1,\ldots, z_{2k}\text{ primitive}}}\beta_{z_1}\overline{\beta_{z_2}}\cdots \beta_{z_{2k-1}}\overline{\beta_{z_{2k}}}.
$$
We now determine the support of $\beta_z'$. If $z = z_1\cdots z_{2k}$ with $z_i\in \RRR_1(Z, \Delta)$ for $1\leq i\leq 2k$, then $Z^{2k}\leq\Norm(z)\leq Z^{2k}(1+\Delta)^{2k}$ and $z\in\RRR_{2k}$. Hence $\beta_z' = 0$ unless $z\in\RRR_{2k}(Z^{2k}, \Delta')$, where $\Delta' = (1+\Delta)^{2k}-1$, and unless the principal ideal $(z)$ has at most $2kf$ prime ideal factors. We now apply the Cauchy-Schwarz inequality to the second factor on the right-hand side of \eqref{bbound1} to get
\begin{equation}\label{bbound3}
\left(\sum_{z}\beta_z'\suma_{w}\chiw(z)\right)^2 \ll \left(\sum_{z}|\beta_z'|^2\right) \cdot \sum_{z}\left|\suma_{w}\chiw(z)\right|^2,
\end{equation}
where the summations over $z$ are implicitly restricted to $z\in\RRR_{2k}(Z^{2k}, \Delta')$. Since $\beta_z'$ is supported on $z$ that are a product of $2k$ numbers $z_i\in\RRR_1(Z, \Delta)$, each of which has at most $f$ prime ideal factors, and since each principal ideal has at most $g$ generators in $\RRR_1(Z, \Delta)$, where $g = 1$ if $D>0$, $g = 4$ if $D = -4$, $g = 6$ if $D = -3$, and $g = 2$ otherwise, we deduce by prime ideal factorization that
$$
|\beta_z'|\leq (2k)!\cdot{2kfg \choose f}^{2k} \leq (2k)!\cdot 2^{24k^2f}.
$$  
Hence the first factor on the right-hand side of \eqref{bbound3} is
\begin{equation}\label{bbound4}
\ll 2^{48k^2f}\Delta'Z^{2k}\ll 2^{48k^2f}\Delta Z^{2k},
\end{equation}
since $\Delta' = (1+\Delta)^{2k}-1\leq 2^{2k}\Delta$. We expand the square in the second factor on the right-hand side of \eqref{bbound3} and rearrange the sums to get
\begin{equation}\label{bbound5}
\suma_{w_1}\suma_{w_2}\sum_{z\in\RRR_{2k}(Z^{2k}, \Delta')}\chi_{w_1}(z)\chi_{w_2}(z).
\end{equation}
For each pair of primitive $w_1$ and $w_2$, set $Y = \Norm(w_1w_2)$ and $r = \Norm(\gcd((w_1), (\ow_2)))$. Using the Lipschitz principle of Davenport \cite{Dav} and Proposition \ref{keycancellation}, we estimate the inner sum by
$$
\sum_{z\in\RRR_{2k}(Z^{2k}, \Delta')}\chi_{w_1}(z)\chi_{w_2}(z) \ll
\begin{cases}
\Delta Z^{2k} + YZ^{k} + Y^2 & \text{if }Y\text{ and }r\text{ are squares} \\
YZ^{k} + Y^2 & \text{otherwise.}
\end{cases}
$$
Hence \eqref{bbound5} is
\begin{equation}\label{bbound6}
\ll \sumsum_{\substack{W < m_1, m_2 \leq (1+\Delta)W \\ m_1m_2\text{ square}}}4^f\left(\Delta Z^{2k} + W^2Z^k + W^4 \right)+(\Delta W)^2\left(W^2Z^k+W^4\right).
\end{equation}
The factor $4^f$ appears because the summations over $w_i$ are implicitly restricted to $w_i$ with at most $f$ prime ideal factors; the number of such $w\in\RRR_1(W)$ satisfying $\Norm(w) = m$ for a given $m$ is at most $2^f$. Combining \eqref{bbound2}, \eqref{bbound4}, and \eqref{bbound6}, we deduce that $|B(W, Z; \alpha, \beta)|^{4k}$ is less than some absolute constant $C>0$ times
$$
2^{50k^2f}\Delta^{4k+1}\left(W^{4k-1}Z^{4k}\log W+ W^{4k+2}Z^{3k} + W^{4k+4}\right),
$$
which implies that $B(W, Z; \alpha, \beta) \ll$
$$
2^{13kf}\Delta^{1+\frac{1}{4k}}\left(W^{-\frac{1}{4k}} + W^{\frac{1}{2k}}Z^{-\frac{1}{4}} + W^{\frac{1}{k}Z^{-1}}\right)WZ \log W.
$$
Note that $W^{-\frac{1}{4k}} > W^{\frac{1}{2k}}Z^{-\frac{1}{4}}$ whenever $Z>W^{\frac{k}{3}}$ and that $W^{-\frac{1}{4k}} > W^{\frac{1}{k}}Z^{-1}$ whenever $Z>W^{\frac{5}{4k}}$. Choosing $k = 3$, we get that 
\begin{equation}\label{bbound7}
B(W, Z; \alpha, \beta) \ll 2^{40f}\Delta^{\frac{13}{12}}W^{\frac{11}{12}}Z\log W 
\end{equation}
whenever $Z>W$. We now exploit the symmetry of the quadratic residue symbol. Indeed, by the law of quadratic reciprocity, we have
$$
\chiw(z) = \delta(w, z)\chi_{z}(w),
$$
where $\delta(w, z)$ depends only on the congruence classes of $w$ and $z$ modulo $8$. Hence
$$
B(W, Z; \alpha, \beta) = \sum_{\omega\bmod 8}\sum_{\zeta\bmod 8} B(W, Z; \alpha(\omega), \beta(\zeta)),
$$
where $\alpha(\omega)_w = \alpha_w\cdot \ONE(w\equiv \omega\bmod 8)$ and $\beta(\zeta)_z = \beta_z\cdot \ONE(z\equiv \zeta\bmod 8)$. But now $B(W, Z; \alpha(\omega), \beta(\zeta)) = \delta(\omega, \zeta) \cdot B(Z, W; \beta(\zeta), \alpha(\omega))$, so
\begin{equation}\label{bbound8}
\begin{array}{rcl}
|B(W, Z; \alpha, \beta)| & \leq & \displaystyle{64^2\cdot \max_{\omega, \zeta\bmod 8} |B(W, Z; \alpha(\omega), \beta(\zeta))|} \\
& = & \displaystyle{64^2\cdot \max_{\omega, \zeta\bmod 8} |B(Z, W; \beta(\zeta), \alpha(\omega))|} \\
& \ll & \displaystyle{2^{40f}\Delta^{\frac{13}{12}}Z^{\frac{11}{12}}W\log Z} 
\end{array}
\end{equation}
whenever $W>Z$. Combining \eqref{bbound7} and \eqref{bbound8}, we get the desired result.
\end{proof}

\subsection{From Proposition~\ref{doubleosc} Proposition~\ref{PropB}}
We will now prove Proposition~\ref{PropB} from Proposition~\ref{doubleosc} by making appropriate choices for the sequences $\{\alpha_{w}\}$ and $\{\beta_{z}\}$. First recall the sum from Proposition \ref{PropB}. We defined
$$
\BB(M, N; f, \Delta) = \sumsum_{\substack{M<p\leq M+\Delta M \\ N<q\leq N+\Delta N}}f(p, q; \br, \bs, \be)
$$
where $\br = (r_1, r_2)\in \{1, 9\}\times \{1, 9\}$, $r_1r_2\equiv 1\bmod 16$ if $d = 8$, $\bs = (s_1, s_2)\in\{\pm 1\}\times \{\pm 1\}$, $s_1s_2 = 1$, $\be\in\FF_2^2$, $\be\neq (0, 0)$, and 
$$
f(p, q; \br, \bs, \be) = 
\begin{cases}
\chi_p(q)^{e_1}\ve(p, q)^{e_2} & \text{if }(p, q)\equiv \br\bmod 16 \text{ and }(\chi_2(p), \chi_2(q)) = \bs\\
0 & \text{otherwise.} 
\end{cases}
$$
For $w, z\in\ZZd$, we set
\begin{align*}
\alpha_{f, w} = \ONE(w\text{ is a prime in }&\ZZd) \cdot \ONE(\Norm(w)\equiv \delta_1\bmod 16) \\
& \cdot \ONE(w\equiv \square\bmod 4\ZZd)\cdot\ONE(\chi_2(\Norm(w)) = s_1)
\end{align*}
and
\begin{align*}
\beta_{f, z} = \ONE(z\text{ is a prime in }&\ZZd) \cdot \ONE(\Norm(z)\equiv \delta_2\bmod 16) \\
& \cdot \ONE(z\equiv \square\bmod 4\ZZd)\cdot\ONE(\chi_2(\Norm(z)) = s_2)
\end{align*}
Since there are exactly two primes in $\ZZd$ lying above each rational prime congruent to $1$ modulo $8$, we have, by \eqref{linkvegamma}, that 
$$
\BB(M, N; f, \Delta) = \frac{1}{2g}\suma_{\substack{w \in\RRR_1(M)}}\suma_{\substack{z \in \RRR_1(N)}}\alpha_{f, w}\beta_{f, z}\left(\frac{\Norm(w)}{\Norm(z)}\right)^{e_1}(\mul(w)\chiw(\oz))^{e_2},
$$
where $g = 4$ if $d_0 = -1$, $g = 2$ if $d_0 = -2$, and $g = 1$ if $d_0 = 2$.

If $e_2 = 0$, then Proposition~\ref{PropB} is a statement about double oscillation of the usual Jacobi symbol $\left(\frac{p}{q}\right)$, and the claim follows from very strong bounds due to Heath-Brown (see \cite{HB}). If $e_2 = 1$, then we apply Proposition~\ref{doubleosc}. Indeed, if $e_1 = 0$ and $e_2 = 1$, we can apply Proposition~\ref{doubleosc} directly to obtain the desired result (absorb $\mul(w)$ into $\alpha_{f, w}$ and note that $\RRR_1(N)$ is invariant under $z\mapsto\oz$). If $e_1 = 1$ and $e_2 = 1$, then by \eqref{linkjacgamma}, we have
$$
\left(\frac{\Norm(w)}{\Norm(z)}\right)\chiw(\oz) = \chiw(z)
$$
whenever $(w)$ and $(z)$ are degree-one primes in $\ZZd$ satisfying $\Norm(w)\equiv\Norm(z)\equiv 1\bmod 8$. Once again, the desired result follows directly from Proposition~\ref{doubleosc} (again absorb $\mul(w)$ into $\alpha_{f, w}$). This finishes the proof of Proposition~\ref{PropB}.

\bibliographystyle{plain}
\bibliography{8rankreferences}

\newpage
\begin{appendix}
\section{Heuristics}\label{Heuristics}
We briefly discuss the conjectural limit of the ratio in Theorem~\ref{mainthm} and the limitations of our methods towards a proof of such a conjecture. Let $G$ be a finite abelian group, and let $\#\Aut(G)$ be the number of automorphims of $G$. Cohen and Lenstra \cite{CohenLenstra} developed a heuristic model for the average structure of class groups of quadratic number fields. Their model is based on the assumption that $G$ occurs as the class group of an imaginary (resp. a real) quadratic field with probability proportional to the inverse of $\#\Aut(G)$ (resp. $\#G\cdot \#\Aut(G)$). Although they stated their model only for the prime-to-$2$ part of the class group, Gerth \cite{Gerth3} extended the model to the $2$-part of the class group by stating that it is $\CL(D)^2$ instead of $\CL(D)$ that behaves like a random group in the sense of~\cite{CohenLenstra}.

Under these assumptions, we can compute a conjectural density for the ratio
$$
\frac{\#\{pq\leq X: rk_4\CL(dpq) = 2, \rk_8\CL(dpq)\geq 1\}}{\#\{pq\leq X: \rk_4\CL(dpq) = 2\}}
$$
from the Main Theorem. Given that $\rk_4\CL(D) = 2$, the $2$-part of $\CL(D)^2$ must be of the form $\ZZ/2^{m}\ZZ\times\ZZ/2^n\ZZ$ for some $n\geq m\geq 1$. In this notation $\rk_8\CL(D) \geq 1$ precisely when $n\geq 2$. An elementary computation yields
$$
\#\Aut(\ZZ/2^{m}\ZZ\times\ZZ/2^n\ZZ) = 
\begin{cases}
3\cdot 2^{4m-3} & \text{if }m = n \\
2^{3m+n-2} & \text{if }m<n.
\end{cases}
$$
Suppose now that $d = -4$, so that we're in the imaginary case. The total weight of all groups of the form $\ZZ/2^{m}\ZZ\times\ZZ/2^n\ZZ$ is
$$
\sum_{m\geq 1}\frac{1}{3\cdot 2^{4m-3}} + \sum_{m\geq 1}\sum_{n \geq m+1}\frac{1}{2^{3m+n-2}} = \frac{4}{9}.
$$
The case when $\rk_8\CL(D) = 0$, i.e. $m = n = 1$, has weight $1/6$. The probability of the complement is thus $(4/9 - 1/6)/(4/9) = 5/8$ and we are led to conjecture
\begin{conjecture}\label{conj1}
Let $d = -4$. Then
$$
\lim_{X\rightarrow\infty}\frac{\#\{pq\leq X: rk_4\CL(dpq) = 2, \rk_8\CL(dpq)\geq 1\}}{\#\{pq\leq X: \rk_4\CL(dpq) = 2\}} = \frac{5}{8}
$$
as $X\rightarrow\infty$.
\end{conjecture}
Similarly, in the real case, when $d = 8$, the total weight is
$$
\sum_{m\geq 1}\frac{1}{3\cdot 2^{6m-3}} + \sum_{m\geq 1}\sum_{n \geq m+1}\frac{1}{2^{4m+2n-2}} = \frac{4}{63},
$$
while the weight of the case $m = n = 1$ is $1/24$. The probability of the complement is then $(4/63-1/24)(4/63) = 11/32$, so that we conjecture
\begin{conjecture}\label{conj2}
Let $d = 8$. Then
$$
\lim_{X\rightarrow\infty}\frac{\#\{pq\leq X: rk_4\CL(dpq) = 2, \rk_8\CL(dpq)\geq 1\}}{\#\{pq\leq X: \rk_4\CL(dpq) = 2\}} = \frac{11}{32}
$$
as $X\rightarrow\infty$.
\end{conjecture}
Both Conjectures \ref{conj1} and \ref{conj2} closely match numerical data generated in Sage.

There is another way to obtain the same conjectures that more closely matches our strategy of proof of Theorem~\ref{mainthm}. For the sake of simplicity, we focus on the case $d = -4$. As we saw in Proposition \ref{algebra}, the existence of an unramified a.f.p.\ $C_8$-extension of $\QQ(\sqrt{-4pq})$ containing $\QQ(\sqrt{-4}, \sqrt{pq})$ is contingent upon two events. The first is
$$
\text{Event }A\text{: condition }\eqref{split22}\text{ holds, the splitting condition at }2,
$$ 
and the second is
$$
\text{Event }B\text{: condition }\eqref{eqp}\text{ holds, the splitting condition at }p.
$$
We already saw in \eqref{samecond} that the splitting condition at $q$ is automatically satisfied if it is satisfied at $p$. Both Events $A$ and $B$ are determined by the values of certain quadratic residue symbols depending on $p$ and $q$. Assuming these symbols take values $+1$ and $-1$ equally often and independently of each other, the probability that both Events $A$ and $B$ occur is $\frac{1}{2}\cdot\frac{1}{2} = \frac{1}{4}$. This is exactly how we prove Theorem~\ref{Thm1}. 

When $\rk_4\CL(-4pq) = 2$, there also exists an unramified a.f.p.\ $C_4$-extension $L'$ (resp. $L''$) of $\QQ(\sqrt{-4pq})$ that contains $\QQ(\sqrt{-4p}, \sqrt{q})$ (resp. $\QQ(\sqrt{p}, \sqrt{-4q})$). For $L'$ (resp. $L''$)to be contained in an unramified a.f.p.\ $C_8$-extension of $\QQ(\sqrt{-4pq})$, there are again two events that must occur. One of them once again concerns the splitting condition at $2$, say Event $A'$ (resp. Event $A''$). The other event, say Event $B'$ (resp. $B''$), concerns the splitting condition at $q$ (resp. $p$). We can once again expect Events $A'$, $A''$, $B'$, and $B''$ to be determined by values of certain quadratic residue symbols, except this time in $\ZZ[\sqrt{-4p}]$ or $\ZZ[\sqrt{-4q}]$. And we can again conjecture that each of there symbols takes the values $+1$ and $-1$ equally often. However, these events are \textit{not} independent. If both $\QQ(\sqrt{-4}, \sqrt{pq})$ and $\QQ(\sqrt{-4p}, \sqrt{q})$ are contained in (distinct) unramified a.f.p.\ $C_8$-extensions of $\QQ(\sqrt{-4pq})$, then so is $\QQ(\sqrt{p}, \sqrt{-4q})$. One can check that out of the events $A$, $A'$, and $A''$, either exactly one or all three events occur, and similarly for $B$, $B'$, and $B''$. Hence, using the principle of inclusion-exclusion, we may conjecture that $\rk_8\CL(D)$ is at least $1$ with probability
\begin{align*}
\BP(A\&B)+ \BP(A'\&B') + \BP(A''\&B'') - &2\BP(A\&A'\&A''\&B\&B'\&B'') \\ 
&=\frac{1}{4}+\frac{1}{4}+\frac{1}{4}-2\cdot \frac{1}{16} = \frac{5}{8}.
\end{align*}     
Thus the discrepancy between our lower bound of $1/4$ from the Main Theorem and the conjectural limit $5/8$ comes from not taking into account unramified a.f.p.\ $C_8$-extensions of $\QQ(\sqrt{-4pq})$ containing $\QQ(\sqrt{-4p}, \sqrt{q})$ or $\QQ(\sqrt{p}, \sqrt{-4q})$. 

The main obstacle in extending the ideas of this paper to handle Events $B'$ and $B''$ is that $\ZZ[\sqrt{-4p}]$ and $\ZZ[\sqrt{-4q}]$ are no longer principal ideal domains, and in fact $\CL(-4p)$ or $\CL(-4q)$ (or both) may have non-trivial odd torsion. Thus it is difficult to control (in a uniform way as $p$ and $q$ vary) the size of the analogues of $\alpha$ from \eqref{alpha}, and a genuinely new idea would be required to apply similar analytic techniques. Theorem \ref{Thm1} already achieves a new lower bound for the $8$-rank, so we leave the task of sharpening this lower bound for a future project. 

\end{appendix}

\end{document}